\documentclass[preprint]{article}

\usepackage[a4paper]{geometry}
\usepackage{amsmath}
\usepackage{amssymb}
\usepackage{amscd}
\usepackage{latexsym}
\usepackage{graphicx}
\usepackage{color}
\usepackage{mathrsfs}
\usepackage{amsthm}
\usepackage[table]{xcolor}
\usepackage{pgfplots,pgfplotstable}
\usepackage{subcaption}

\def\R {\mathbb R}

\def\Z {\mathbb Z}

\renewcommand{\div}{\mathop{\hbox{\rm div}}}

\newcommand{\curl}{\mathop{\hbox{\rm curl}}}

\newcommand{\sign}{\mr{sign}}

\newcommand{\lk}{\ell\raisebox{-2pt}{$\!\kappa$}}

\renewcommand{\o}{\mr{o}}

\newtheorem{theo}{Theorem}
\newtheorem{lemma}[theo]{Lemma}
\newtheorem{defin}[theo]{Definition}
\newtheorem{alg}{Algorithm}

\newcommand{\mb}{\mathbf}
\newcommand{\bs}{\boldsymbol}
\newcommand{\lra}{\longrightarrow}
\newcommand{\mr}{\mathrm}
\newcommand{\mc}{\mathcal}

\newcommand{\mk}{\mathfrak}

\newcommand{\T}{\mc{T}}
\newcommand{\vv}{\mb{v}}
\newcommand{\E}{\mc{E}}
\newcommand{\F}{\mc{F}}
\newcommand{\G}{\mc{G}}
\newcommand{\A}{\mc{A}}
\newcommand{\B}{\mc{B}}
\newcommand{\K}{\mc{K}}

\newcommand{\cb}{\sigma_{_{\!\B'}}}%{\sigma_{_{\!\scriptstyle\B'}}}

\begin{document}

\title{Geometric construction of bases of $H_2(\overline \Omega, \partial \Omega; \Z)$}

\author{Ana Alonso Rodr\'{i}guez\thanks{Dipartimento di Matematica, Universit\`a di Trento, 38123 Povo (Trento), Italy} \and Enrico Bertolazzi \thanks{Dipartimento di Ingegneria Industriale, Universit\`a di Trento, 38123 Mesiano (Trento), Italy} \and Riccardo Ghiloni\footnotemark[1] \and Ruben Specogna\thanks{Universit\`a di Udine, Dipartimento di Ingegneria Elettrica, Gestionale e
Meccanica, Via delle Scienze 208, 33100 Udine, Italy}}

%\date{}

\maketitle

\begin{abstract}
We present an efficient algorithm for the construction of a basis of $H_2(\overline{\Omega},\partial\Omega;\Z)$ via the Poincar\'e-Lefschetz duality theorem. Denoting by $g$ the first Betti  number of $\overline \Omega$ the idea is to find, first $g$ different $1$-boundaries of $\overline{\Omega}$ with supports contained in $\partial\Omega$ whose homology classes in $\R^3 \setminus \Omega$ form a basis of $H_1(\R^3 \setminus \Omega;\Z)$, and then to construct in $\overline{\Omega}$ a homological Seifert surface of each one of these $1$-boundaries. The Poincar\'e-Lefschetz duality theorem ensures that the relative homology classes of these homological Seifert surfaces in $\overline\Omega$ modulo $\partial\Omega$ form a basis of $H_2(\overline\Omega,\partial\Omega;\Z)$. We devise a simply procedure for the construction of the required set of $1$-boundaries of $\overline{\Omega}$ that, combined with a fast algorithm for the construction of homological Seifert surfaces, allows the efficient computation of a basis of $H_2(\overline{\Omega},\partial\Omega;\Z)$ via this very natural geometrical approach. Some numerical experiments show the efficiency of the method and its performance comparing with other algorithms.
\end{abstract}

%%%%%%%

\section{Introduction}

%%***** REFERENCES *****

% We present an efficient algorithm for the construction of a basis of $H_2(\overline{\Omega},\partial\Omega;\Z)$. There are many

% Another application of cohomology computations concerns computer graphics and geometry processing algorithms (see \cite{VBW13}). *** Svilupare a decidere dove va ***).

Consider a bounded domain $\Omega$ of $\R^3$ whose closure $\overline{\Omega}$ is polyhedral and whose boundary $\partial \Omega$ is sufficiently regular, like that used for finite element approximation of differential problems. Our aim is to develop a set of fast and robust algorithms for the automatic identification and construction of that homological structures that influence the solvability of differential problems defined on $\Omega$. Let us consider, for instance,  the curl-div system
\[
\begin{array}{ll}
\curl {\bf u} = {\bf F} & \hbox{in } \Omega \\
\div {\bf u} = G & \hbox{in } \Omega \\
{\bf u} \cdot n = g & \hbox{on } \partial \Omega
\end{array}
\]
It is well-know that the solution of this problem is not unique if $g$, the first Betti number of $\Omega$, is greater than zero. Two different ways to fix a unique solution are to prescribe the circulation around a set of $1$-cycles in $\Omega$ that are representatives of a basis of the first homology group $H_1(\overline{\Omega};\Z)$ of $\overline{\Omega}$ or to prescribe the flux through a set of surfaces that are representatives of a basis of the second relative homology group $H_2(\overline{\Omega},\partial\Omega;\Z)$ of $\overline{\Omega}$ modulo $\partial\Omega$.

Let us consider a triangulation of $\overline{\Omega}$; namely, a tetrahedral mesh of $\overline{\Omega}$. The incidence matrices of such a triangulation, tetrahedra-to-faces, faces-to-edges and edges-to-vertices, are the integer matrix representations of the so-called boundary operators associated with the given triangulation. The standard procedure to compute the homology and cohomology groups of $\overline{\Omega}$ is based on the computation of the Smith normal form of these integer matrices, a computationally demanding algorithm even in the case of sparse matrices (see e.g. \cite{Mun84} and \cite{Ilio89,DSV01}). Thus, before the Smith normal form procedure is employed, the problem size is reduced using fast algorithms (usually algorithms that run in linear time) that remove homologically irrelevant parts of the triangulation (see e.g. \cite{Spec5_10}, \cite{MB09}). An implementation of these techniques have been integrated in the finite element mesh generator {\em Gmesh} by Pellikka {\em et al.} (see  \cite{PSKG13}). Other software that perform homology and cohomology computations, with less emphasis on finite element modeling, are {\em CHomP}~\cite{CHomP}, jPlex~\cite{jPlex} and {\em GAP homology}~\cite{GAP}.
A different approach, using chain contraction instead of the classical reduction algorithms, is described in \cite{PR15}, the computational cost is higher but it has more functionalities, since it provides more comprehensive homological information.

If the goal is to construct a basis of $H_2(\overline{\Omega},\partial\Omega;\Z)$, specific algorithms could be more efficient that generic algorithms for the computation of homology and cohomology groups.

A specific approach for the construction of a basis of $H_2(\overline{\Omega},\partial\Omega;\Z)$ has been proposed by Kotiuga in \cite{Kot87}, \cite{Kot88}, \cite{Kot89} and \cite{GK04}. There the aim is to construct the so-called {\em ``cuts'' } of $\Omega$; namely, surfaces-with-boundary $\{S_i\}_{i=1}^g$ of $\overline{\Omega}$ with $\partial S_i \subset \partial \Omega$ which permit to construct a single-valued magnetic scalar potential in $\Omega \setminus \bigcup_{i=1}^g S_i$ of any given current density in $\Omega$. These cuts are nonsingular polyhedral representatives of a basis of $H_2(\overline{\Omega},\partial\Omega;\Z)$. The algorithm consists in two main steps. Starting with a basis of $H_1(\overline{\Omega};\Z)$, in the first step, one constructs a basis $\{f_i\}_{i=1}^g$ of the cohomology group $H^1(\overline{\Omega};\Z)$ approximating a differential problem with a finite element method. Then the second step is to construct the cuts of $\Omega$ as level sets of the maps $\{ f_i\}_{i=1}^g$. The representatives of the basis are regular surfaces and this justify the substantial  complexity of the procedure.
%\footnote{\color{red}{Si potrebbe forse aggiungere qualche commento sulla complessit\`a di questo algoritmo.}}

In this paper we focus on the construction of a basis of $H_2(\overline{\Omega},\partial\Omega;\Z)$ using a geometric approach based on the Poincar\'e-Lefschetz duality theorem. Here we are not interested in questions concerning regularity. Indeed the representatives of the basis that we construct are formal linear combinations (with integer coefficients) of oriented faces of the given triangulation that we call homological Seifert surfaces. This allows to gain in efficiency from the computational point of view.

Let us precise what we meant when we said that the boundary $\partial \Omega$ of $\Omega$ is sufficiently regular. In what follows we will assume that $\partial \Omega$ is {\em locally flat}; that is, for every point $x \in \partial \Omega$, there exist an open neighborhood $U_x$ of $x$ in $\R^3$ and a homeomorphism $\phi_x:U_x \lra \R^3$ such that $\phi_x(U_x \cap \partial \Omega)=P$, where $P$ is the coordinate plane $\{(x,y,z) \in\R^3 \, | \, z=0\}$  (see \cite{Brown62, BFG10}). This kind of domains includes all Lipschitz polyhedral domains, but also domains like the crossed bricks (see, e.g., Fig. 3.1 in \cite{Monk}). Let $\T$ be a triangulation of $\overline{\Omega}$.
A $1$-cycle $\gamma$ of $\T$ is a formal linear combination (with integer coefficients) of oriented edges of $\T$ with zero boundary. The $1$-cycle $\gamma$ is said to be a $1$-boundary of $\T$ if it is equal to the boundary of a formal linear combination $S$ of oriented faces of $\T$. If such a $S$ exists, we call it {\em homological Seifert surface of $\gamma$ in $\T$}.

Given $g$ different $1$-boundaries $\{ \sigma'_n \}_{n=1}^g$ of $\T$ with supports contained in $\partial\Omega$ and whose homology classes in $\R^3 \setminus \Omega$ form a basis of $H_1(\R^3 \setminus \Omega;\Z)$, and given for each $n=1,\dots,g$ a homological Seifert surface $S_n$ of $\sigma'_n$ in $\T$, the Poincar\'e-Lefschetz duality theorem ensures that the relative homology classes of the surfaces $\{ S_n\}_{n=1}^g$ in $\overline\Omega$ modulo $\partial\Omega$ form a basis of $H_2(\overline\Omega,\partial\Omega;\Z)$.
In \cite{ABGS15} we propose and analyze a very efficient algorithm that, given a $1$-boundary $\gamma$ of $\T$, computes a homological Seifert surfaces of $\gamma$ in $\T$. Hence this algorithm allows the construction of a basis of $H_2(\overline\Omega,\partial\Omega;\Z)$ once we know a set of $1$-boundaries $\sigma'_1,\ldots,\sigma'_g$ of $\T$  with supports contained in $\partial\Omega$ and whose homology classes in $\R^3 \setminus \Omega$ form a basis of $H_1(\R^3 \setminus \Omega;\Z)$.

If $\partial \Omega$ is connected, an algorithm for the construction of such a set of $1$-boundaries have been analyzed in \cite{HO02}. The first step is to construct a set of $2g$ 1-cycles $\{\gamma_l \}_{l=1}^{2g}$ of $\partial \Omega$ that are representatives of a basis of $H_1(\partial \Omega;\Z)$. The second step is to compute $g$ linear combinations $\{\widehat \sigma_n = \sum_{l=1}^{2g} B_{n,l} \gamma_l\}_{n=1}^g$ of these $2g$ $1$-cycles $\gamma_l$, whose homology classes in $\R^3 \setminus \Omega$ form a basis of the homology group $H_1(\R^3 \setminus \Omega; \Z)$. If $\partial \Omega$ is connected, the $1$-cycles $\widehat \sigma_n$ of $\partial\Omega$ turn out to be also $1$-boundaries of $\overline{\Omega}$ so we can take $\sigma'_n = \widehat \sigma_n$ for $n=1,\dots,g$.

 In \cite{ABGV13} the authors extend to the case of a non connected boundary $\partial \Omega$ the construction of representatives of a basis of $H_1(\partial \Omega; \Z)$ and then the construction of $g$ independent linear combinations of these 1-cycles that are representatives of a basis of $H_1(\R^3 \setminus \Omega; \Z)$. But, being $\partial \Omega$ not connected, the elements of this basis of $H_1(\R^3 \setminus \Omega; \Z)$ are not necessarily $1$-boundaries of $\overline{\Omega}$.

For instance in Figure~\ref{fig:intro1}(a) the domain $\Omega$ is an open solid torus with a coaxial smaller closed solid torus removed and the homology class of the $1$-cycle $\widehat{\sigma}_1$ of $\partial \Omega$, represented by a continuous line, is different from zero in $H_1(\R^3 \setminus \Omega; \Z)$ (indeed it is a representative of an element of a basis of $H_1(\R^3 \setminus \Omega; \Z)$), but $\widehat{\sigma}_1$ is not a $1$-boundary of $\overline{\Omega}$. To obtain a $1$-boundary $\sigma'_1$ of $\overline{\Omega}$ homologous to $\widehat{\sigma}_1$ in $\overline{\Omega}$, we need to add a $1$-cycle $\sigma^*_1$ of $\partial \Omega$, like the one represented by the dotted line: $\sigma'_1:=\widehat{\sigma}_1+\sigma^*_1$. Now $\sigma'_1$ is the boundary of the homological Seifert surface $S_1$ of $\overline{\Omega}$ represented in Figure~\ref{fig:intro1}(a). Analogously the homology class of the $1$-cycle $\widehat{\sigma}_2$ of $\partial \Omega$, represented by a continuous line in Figure~\ref{fig:intro1}(b) , is different from zero in $H_1(\R^3 \setminus \Omega; \Z)$. Together with $\widehat{\sigma}_1$ they are representatives of a  basis of $H_1(\R^3 \setminus \Omega; \Z)$. However $\widehat \sigma_2$  is not a $1$-boundary of $\overline{\Omega}$ and to obtain a $1$-boundary $\sigma'_2$ of $\overline{\Omega}$ homologous to $\widehat{\sigma}_2$ in $\overline{\Omega}$, we need to add a $1$-cycle $\sigma^*_2$ of $\partial \Omega$, like the one represented by the dotted line: $\sigma'_2:=\widehat{\sigma}_2+\sigma^*_2$ is the boundary of the homological Seifert surface $S_2$ of $\overline{\Omega}$ represented in Figure~\ref{fig:intro1}(b).

\begin{figure}[!htb]
\centering
 \includegraphics[width=.75\textwidth]{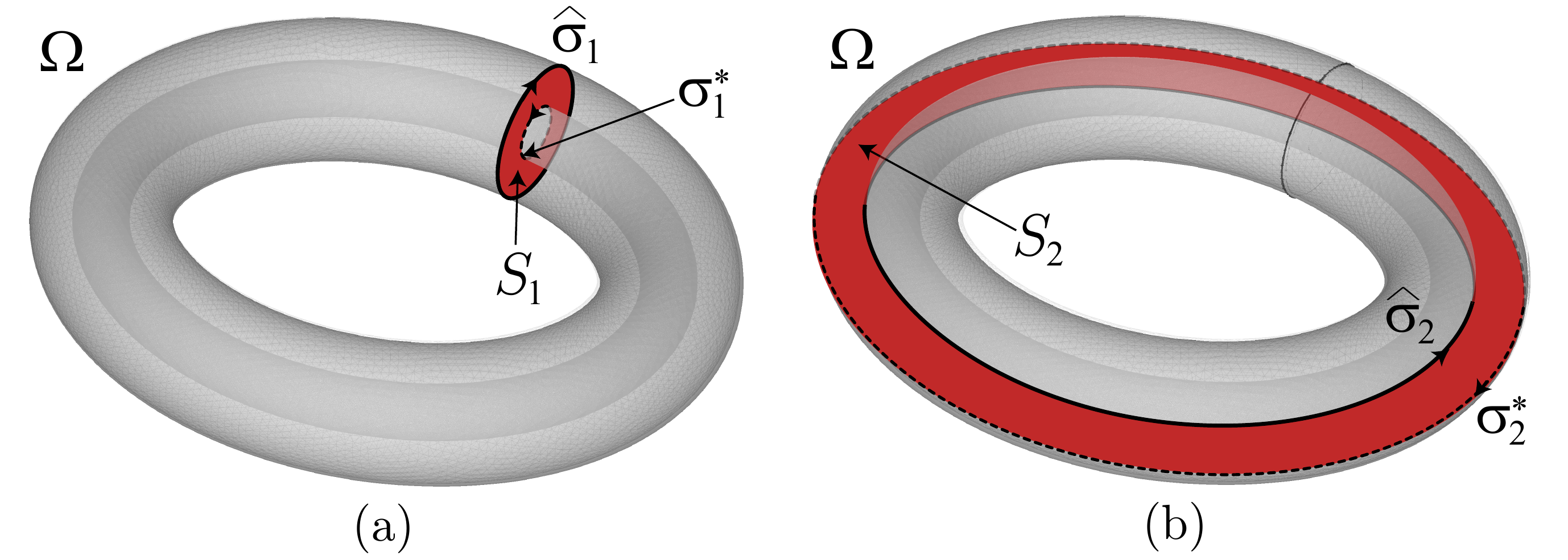}
  \caption{The boundaries.} \label{fig:intro1}
\end{figure}

The main theoretical result of this paper is, starting from a set of $2g$ $1$-cycles of $\partial \Omega$ representing a basis of $H_1(\partial \Omega; \Z)$, to identify a set of $g$ {\em $1$-boundaries of $\overline{\Omega}$} whose homology classes in $\R^3 \setminus \Omega$ form a basis of $H_1(\R^3 \setminus \Omega;\Z)$. This is done in Section 2. In Section 3 we make precise some implementation issues concerning the effective construction of the mentioned $1$-boundaries. Moreover, for the sake of completeness, we briefly describe the algorithm for the construction of homological Seifert surfaces studied in \cite{ABGS15}. The main tool in both cases is the closed block dual barycentric complex of a triangulation.
Combining this two procedures we obtain an algorithm for the construction of a basis of $H_2(\overline{\Omega},\partial\Omega;\Z)$. Finally in Section 4 we present some numerical results illustrating the robustness and efficiency of this geometrical approach. We include also some comparisons with the results obtained using the cohomology solver integrated in {\em Gmsh}.

%Let $\T=(V,E,F,K)$ be a finite triangulation of $\overline \Omega$ where $V$ is the set of vertices, $E$ the set of edges, $F$ the set of faces and $K$ the set of tetrahedra of $\T$.  Let $\T_\partial =(V_\partial, E_\partial, F_\partial)$ be the triangulation of $\partial \Omega$ induced by $\T$; namely we have that $V_\partial=V \cap \partial \Omega$, $E_\partial$ is the set of edges of $\T$ with both vertices in $V_\partial$ and $F_\partial$ is the set of faces with all vertices in $V_\partial$.

%The paper is organized as follows. In Section 2 we study how to construct $1$-boundaries $\sigma'_1,\ldots,\sigma'_g$ of $\T$ whose homology classes in $\R^3 \setminus \Omega$ form a basis of $H_1(\R^3 \setminus \Omega;\Z)$ and with supports contained in $\partial\Omega$ even if $\partial\Omega$ is not connected. For the sake of completeness in Section 3 we briefly describe the algorithm for the construction of homological Seifert surfaces studied in \cite{ABGS15}. Finally in Section 4 we present some numerical results illustrating the robustness and efficiency of this geometrical approach.

\section{The construction of the $1$-boundaries}

Let $\T=(V,E,F,K)$ be a finite triangulation of $\overline \Omega$ where $V$ is the set of vertices, $E$ the set of edges, $F$ the set of faces and $K$ the set of tetrahedra of $\T$.  Let $\T_\partial =(V_\partial, E_\partial, F_\partial)$ be the triangulation of $\partial \Omega$ induced by $\T$; namely we have that $V_\partial=V \cap \partial \Omega$, $E_\partial$ is the set of edges of $\T$ with both vertices in $V_\partial$ and $F_\partial$ is the set of faces with all vertices in $V_\partial$.

As indicated in the introduction, if $\partial\Omega$ is connected, then the desired $1$-boundaries $\sigma'_m$ are constructed in \cite{HO02}. More precisely, under this connectedness condition the authors construct $1$-cycles $\sigma_1,\ldots,\sigma_g,\widehat{\sigma}_1,\ldots,\widehat{\sigma}_g$ of $\T_\partial$ in such a way that their homology classes in $\T_\partial$ form a basis of $H_1(\T_\partial;\Z)$ and it holds:
\begin{itemize}
 \item $\sigma_1,\ldots,\sigma_g$ bounds in $\R^3 \setminus \Omega$ and their homology classes in $\Omega$ form a basis of $H_1(\Omega;\Z)$,
 \item $\widehat{\sigma}_1,\ldots,\widehat{\sigma}_g$ bounds in $\Omega$ and their homology classes in $\R^3 \setminus \Omega$ form a basis of $H_1(\R^3 \setminus \Omega;\Z)$.
\end{itemize}
By defining $\sigma'_m:=\widehat{\sigma}_m$ for every $m \in \{1,\ldots,g\}$, we are done.
We now consider the more complicated case in which $\partial\Omega$ is not connected.

Let us recall some results from Section 6 of \cite{ABGV13}. As we have said, $\Gamma_0,\Gamma_1,\ldots,\Gamma_p$ denote the connected components of $\partial\Omega$. By the Jordan separation theorem, each open subset $\R^3 \setminus \Gamma_r$ of $\R^3$ has two connected components, both having $\Gamma_r$ as boundary. Denote by $D_r$ the bounded connected component of $\R^3 \setminus \Gamma_r$ and by $g_r$ the first Betti number of its closure $\overline{D_r}$ in $\R^3$. Rearranging the indices $r \in \{0,1,\ldots,p\}$ if necessary, we can suppose that $\Gamma_0$ is the ``external'' component of $\partial\Omega$; namely, it holds: $\overline{\Omega}=\overline{D_0} \setminus \bigcup_{r=1}^pD_r$ and hence $\R^3 \setminus \Omega=(\R^3 \setminus D_0) \cup \bigcup_{r=1}^p\overline{D_r}$. The reader reminds that $H_1(\partial\Omega;\Z)$ is isomorphic to $\bigoplus_{r=0}^pH_1(\Gamma_r;\Z)$, so we have that $2g=\sum_{r=0}^p2g_r$ or, equivalently, $g=\sum_{r=0}^pg_r$.

For convenience, if $c$ is a $1$-cycle of $\R^3$ with support contained in a subset $Z$ of $\R^3$, then we denote by $[c]_Z$ the homology class of $c$ in $Z$.

For every $r \in \{0,1,\ldots,p\}$, $\partial D_r=\Gamma_r$ is connected, so, as we said above, we can construct $1$-cycles $\{\sigma_{r,s}\}_{s=1}^{g_r} \cup \{\widehat{\sigma}_{r,s}\}_{s=1}^{g_r}$ of $\T_\partial$ with support contained in $\Gamma_r$ such that:
\begin{equation} \label{eq:1}
\text{$\big\{[\sigma_{r,1}]_{\overline{D_r}}, \ldots,[\sigma_{r,g_r}]_{\overline{D_r}}\big\}$ is a basis of $H_1(\overline{D_r};\Z)$},
\end{equation}
\begin{equation} \label{eq:2}
\text{$[\sigma_{r,s}]_{\R^3 \setminus D_r}=0 \;$ for every $s \in \{1,\ldots,g_r\}$}
\end{equation}
and
\begin{equation} \label{eq:3}
\text{$\big\{[\widehat{\sigma}_{r,1}]_{\R^3 \setminus D_r}, \ldots,[\widehat{\sigma}_{r,g_r}]_{\R^3 \setminus D_r}\big\}$ is a basis of $H_1(\R^3 \setminus D_r;\Z)$},
\end{equation}
\begin{equation} \label{eq:4}
\text{$[\widehat{\sigma}_{r,s}]_{\overline{D_r}}=0 \;$ for every $s \in \{1,\ldots,g_r\}$}.
\end{equation}

It follows that
\begin{equation} \label{eq:basis-inside}
%\text{the $1$-cycles $\big\{\sigma_{0,s}\big\}_{s=1}^{g_0} \cup \big\{\widehat{\sigma}_{1,s}\big\}_{s=1}^{g_1} \cup \ldots \cup \big\{\widehat{\sigma}_{p,s}\big\}_{s=1}^{g_p}$ induces a basis of $H_1(\overline{\Omega};\Z)$}
\text{$\big\{[\sigma_{0,s}]_{\overline{\Omega}}\big\}_{s=1}^{g_0} \cup \big\{[\widehat{\sigma}_{1,s}]_{\overline{\Omega}}\big\}_{s=1}^{g_1} \cup \ldots \cup \big\{[\widehat{\sigma}_{p,s}]_{\overline{\Omega}}\big\}_{s=1}^{g_p}$ is a basis of $H_1(\overline{\Omega};\Z)$}
\end{equation}
and
\begin{equation} \label{eq:basis-outside}
\text{$\big\{[\widehat{\sigma}_{0,s}]_{\R^3 \setminus \Omega}\big\}_{j=s}^{g_0} \cup \big\{[\sigma_{1,s}]_{\R^3 \setminus \Omega}\big\}_{s=1}^{g_1} \cup \ldots \cup \big\{[\sigma_{p,s}]_{\R^3 \setminus \Omega}\big\}_{s=1}^{g_p}$ is a basis of $H_1(\R^3 \setminus \Omega;\Z)$,}
\end{equation}
where the homology classes corresponding to $\Gamma_r$ are omitted if $g_r=0$. Since $\R^3 \setminus \Omega$ is equal to the disjoint union $(\R^3 \setminus D_0) \cup \bigcup_{r=1}^p\overline{D_r}$,  \eqref{eq:basis-outside} follows immediately from \eqref{eq:3} with $r=0$ and \eqref{eq:1} with $r \geq 1$. For a proof of \eqref{eq:basis-inside}, we refer the reader to Theorem 3.2.2.1 of \cite{DG98} or to Theorem 6 of \cite{ABGV13}.

The problem is now that we do not know if the $1$-cycles $\big\{\widehat{\sigma}_{0,s}\big\}_{s=1}^{g_0} \cup \big\{\sigma_{1,s}\big\}_{s=1}^{g_1} \cup \ldots \cup \big\{\sigma_{p,s}\}_{s=1}^{g_p}$ of $\T$ are $1$-boundaries of $\T$. Our idea to overcome this difficulty is to replace each $1$-cycle with a $1$-boundary without changing its homology class in $H_1(\R^3 \setminus \Omega;\Z)$.

Let $P:=\{1,\ldots,p\}$, let $\Omega_0:=\R^3 \setminus \bigcup_{i \in P}\overline{D_i}=(\R^3 \setminus D_0) \cup \Omega$ and, for every $r \in P$, let $P_r:=P \setminus \{r\}$ and let $\Omega_r:=D_0 \setminus \bigcup_{i \in P_r}\overline{D_i}=\Omega \cup \overline{D_r}$. Let us observe that
\begin{equation} \label{eq:basis-inside-0}
\text{$\bigcup_{i \in P}\big\{[\widehat{\sigma}_{i,s}]_{\overline{\Omega_0}}\big\}_{s=1}^{g_i} \;$ is a basis of $H_1(\overline{\Omega_0};\Z)$}
\end{equation}
and
\begin{equation} \label{eq:basis-inside-r}
\text{$\big\{[\sigma_{0,s}]_{\overline{\Omega_r}}\big\}_{s=1}^{g_0} \cup \bigcup_{i \in P_r }\big\{[\widehat{\sigma}_{i,s}]_{\overline{\Omega_r}}\big\}_{s=1}^{g_i} \;$ is a basis of $H_1(\overline{\Omega_r};\Z)$}
\end{equation}
for every $r \in P$. Assertion \eqref{eq:basis-inside-r} follows immediately by applying \eqref{eq:basis-inside} with $\Omega$ equal to $\Omega_r$. Let $B$ be an open ball of $\R^3$ containing $\bigcup_{i \in P}\overline{D_i}$, then, by applying \eqref{eq:basis-inside} with $\Omega$ equal to $B^*:=B \setminus \bigcup_{i \in P}\overline{D_i}$, we infer that $\bigcup_{r \in P}\big\{[\widehat{\sigma}_{r,s}]_{\overline{B^*}}\big\}_{s=1}^{g_r}$ is a basis of $H_1(\overline{B^*};\Z)$. Since $\overline{B^*}$ is a strong deformation retract of $\overline{\Omega_0}$, we obtain at once \eqref{eq:basis-inside-0}. Since $\overline{\Omega_r}$ is equal to the disjoint union $\bigcup_{i \in P}\overline{D_i}$ if $r=0$ and $(\R^3 \setminus D_0) \cup \bigcup_{i \in P_r}\overline{D_i}$ if $r \in P$, we have also that
\begin{equation} \label{eq:basis-outside-0}
\text{$\bigcup_{i \in P}\big\{[\sigma_{i,s}]_{\R^3 \setminus \Omega_0}\big\}_{s=1}^{g_i} \;$ is a basis of $H_1(\R^3 \setminus \Omega_0;\Z)$}
\end{equation}
and
\begin{equation} \label{eq:basis-outside-r}
\text{$\big\{[\widehat{\sigma}_{0,s}]_{\R^3 \setminus \Omega_r}\big\}_{s=1}^{g_0} \cup \bigcup_{i \in P_r }\big\{[\sigma_{i,s}]_{\R^3 \setminus \Omega_r}\big\}_{s=1}^{g_i} \;$ is a basis of $H_1(\R^3 \setminus \Omega_r;\Z)$}
\end{equation}
for every $r \in P$.

For every $s \in \{1,\ldots,g_0\}$, the support of $\widehat{\sigma}_{0,s}$ is contained in $\Gamma_0 \subset \overline{\Omega_0}$. In this way, thanks to \eqref{eq:basis-inside-0}, there exist, and are unique, integers $\{\alpha^{0,s}_{i,j}\}_{i,j}$ such that
\begin{equation} \label{eq:0s}
[\widehat{\sigma}_{0,s}]_{\overline{\Omega_0}}=\sum_{i \in P} \sum_{j=1}^{g_i}\alpha^{0,s}_{i,j} \, [\widehat{\sigma}_{i,j}]_{\overline{\Omega_0}} \; .
\end{equation}
Similarly, for every $r \in P$ and for every $s \in \{1,\ldots,g_r\}$, the support of $\sigma_{r,s}$ is contained in $\Gamma_r \subset \overline{\Omega_r}$. In this way, thanks to \eqref{eq:basis-inside-r}, there exist, and are unique, integers $\{\alpha^{r,s}_{i,j}\}_{i,j}$ such that
\begin{equation} \label{eq:rs}
[\sigma_{r,s}]_{\overline{\Omega_r}}=\sum_{j=1}^{g_0}\alpha^{r,s}_{0,j} \, [\sigma_{0,j}]_{\overline{\Omega_r}}+\sum_{i \in P_r} \sum_{j=1}^{g_i}\alpha^{r,s}_{i,j} \, [\widehat{\sigma}_{i,j}]_{\overline{\Omega_r}} \; .
\end{equation}

Define the $1$-cycles $\big\{\widehat{\sigma}'_{0,s}\big\}_{s=1}^{g_0} \cup \big\{\sigma_{1,s}'\big\}_{s=1}^{g_1} \cup \ldots \cup \big\{\sigma_{p,s}'\big\}_{s=1}^{g_p}$ of $\T_\partial$ by setting
\begin{equation} \label{eq:sigma'0s}
\widehat{\sigma}_{0,s}':=\widehat{\sigma}_{0,s}-\sum_{i \in P} \sum_{j=1}^{g_i}\alpha^{0,s}_{i,j} \, \widehat{\sigma}_{i,j}
\end{equation}
for every $s \in \{1,\ldots,g_0\}$, and
\begin{equation} \label{eq:sigma'rs}
\sigma_{r,s}':=\sigma_{r,s}-\sum_{j=1}^{g_0}\alpha^{r,s}_{0,j} \, \sigma_{0,j}-\sum_{i \in P_r} \sum_{j=1}^{g_i}\alpha^{r,s}_{i,j} \, \widehat{\sigma}_{i,j}
\end{equation}
for every $r \in P$ and for every $s \in \{1,\ldots,g_r\}$.

\begin{theo} \label{thm:H2}
The $1$-cycles  %$\big\{\widehat{\sigma}'_{0,s}\big\}_{s=1}^{g_0} \cup \big\{\sigma_{1,s}'\big\}_{s=1}^{g_1} \cup \ldots \cup \big\{\sigma_{p,s}'\big\}_{s=1}^{g_p}$
of $\T_\partial$ defined in \eqref{eq:sigma'0s} and in \eqref{eq:sigma'rs} have the following properties:
\begin{itemize}
 \item[$(1)$] They are $1$-boundaries of $\T$; namely, their homology classes in $\overline{\Omega}$ are null.
 \item[$(2)$] $[\widehat{\sigma}'_{0,s}]_{\R^3 \setminus \Omega}=[\widehat{\sigma}_{0,s}]_{\R^3 \setminus \Omega} \;$ for every $s \in \{1,\ldots,g_0\}$ and $[\sigma_{r,s}']_{\R^3 \setminus \Omega}=[\sigma_{r,s}]_{\R^3 \setminus \Omega} \;$ for every $r \in \{1,\ldots,p\}$ and for every $s \in \{1,\ldots,g_r\}$. In particular, the set
 \[
 \big\{[\widehat{\sigma}'_{0,s}]_{\R^3 \setminus \Omega}\big\}_{s=1}^{g_0} \cup \big\{[\sigma_{1,s}']_{\R^3 \setminus \Omega}\big\}_{s=1}^{g_1} \cup \ldots \cup \big\{[\sigma_{p,s}']_{\R^3 \setminus \Omega}\big\}_{s=1}^{g_p}
 \]
 is a basis of $H_1(\R^3 \setminus \Omega;\Z)$.
 \item[$(3)$] Let $S_{0,s}$ be a homological Seifert surface of $\widehat{\sigma}_{0,s}'$ for every $s \in \{1,\ldots,g_0\}$ and let $S_{r,s}$ be a homological Seifert surface of $\sigma_{r,s}'$ for every $r \in \{1,\ldots,p\}$ and for every $s \in \{1,\ldots,g_r\}$. Then the homology classes of such surfaces $\big\{S_{r,s}\big\}_{r \in \{0,1,\ldots,p\}, s \in \{1,\ldots,g_r\}}$ in $\overline{\Omega}$ modulo $\partial\Omega$ form a basis of $H_2(\overline{\Omega},\partial\Omega;\Z)$.
\end{itemize}
\end{theo}
\begin{proof}
$(1)$ Let $s \in \{1,\ldots,g_0\}$. We must prove that $[\widehat{\sigma}_{0,s}']_{\overline{\Omega}}=0$.  Observe that $\R^3=\overline{D_0} \cup \overline{\Omega_0}$ and $\overline{\Omega}=\overline{D_0} \cap \overline{\Omega_0}$. In this way, the Mayer-Vietoris exact sequence associated with the splitting $\R^3=\overline{D_0} \cup \overline{\Omega_0}$ implies that the following inclusion homomorphism is an isomorphism:
\[
i_* \oplus j_*:H_1(\overline{\Omega};\Z) \lra H_1(\overline{D_0};\Z) \oplus H_1(\overline{\Omega_0};\Z),
\]
where $i_*$ and $j_*$ are the homomorphisms induced by the inclusions $i:\overline{D_0} \hookrightarrow \R^3$ and $j:\overline{\Omega_0} \hookrightarrow \R^3$, respectively. It follows that $[\widehat{\sigma}_{0,s}']_{\overline{\Omega}}=0$ if and only if $[\widehat{\sigma}_{0,s}']_{\overline{D_0}}=i_*([\widehat{\sigma}_{0,s}']_{\overline{\Omega}})=0$ and $[\widehat{\sigma}_{0,s}']_{\overline{\Omega_0}}=j_*([\widehat{\sigma}_{0,s}']_{\overline{\Omega}})=0$. By \eqref{eq:0s} and \eqref{eq:sigma'0s}, we have that $[\widehat{\sigma}_{0,s}']_{\overline{\Omega_0}}=0$. Since $\overline{D_r} \subset \overline{D_0}$ for every $r \in P$, equality \eqref{eq:4} ensures that $[\widehat{\sigma}_{i,j}]_{\overline{D_0}}=0$ for every $i \in P$ and for every $j \in \{1,\ldots,g_i\}$. In this way, by \eqref{eq:sigma'0s}, we infer that $[\widehat{\sigma}'_{0s}]_{\overline{D_0}}=0$. This proves that $[\widehat{\sigma}'_{0,s}]_{\overline{\Omega}}=0$, as desired.

For any given $r \in P$ and $s \in \{1,\ldots,g_r\}$, the proof of the fact that  $[\sigma'_{r,s}]_{\overline{\Omega}}=0$ is similar. One must consider the Mayer-Vietoris sequence associated with splitting $\R^3=(\R^3 \setminus D_r) \cup \overline{\Omega_r}$, points \eqref{eq:rs} and \eqref{eq:sigma'rs}, and the inclusions $\R^3 \setminus D_0 \subset \R^3 \setminus D_r$ and $\overline{D_i} \subset \R^3\setminus D_r$ for every $i \in P_r$, together with equalities \eqref{eq:2} and \eqref{eq:4}.

$(2)$ Since $\R^3 \setminus D_0 \subset \R^3 \setminus \Omega$ and $\overline{D_i} \subset \R^3 \setminus \Omega$ for every $i \in P$, equalities \eqref{eq:2} and \eqref{eq:4} imply that $[\sigma_{0,j}]_{\R^3 \setminus \Omega}=0$ for every $j \in \{1,\ldots,g_0\}$ and $[\widehat{\sigma}_{i,j}]_{\R^3 \setminus \Omega}=0$ for every $i \in P$ and for every $j \in \{1,\ldots,g_i\}$. By \eqref{eq:sigma'0s} and \eqref{eq:sigma'rs}, we have that $[\widehat{\sigma}'_{0,s}]_{\R^3 \setminus \Omega}=[\widehat{\sigma}_{0,s}]_{\R^3 \setminus \Omega}$ for every $s \in \{1,\ldots,g_0\}$ and $[\sigma_{r,s}']_{\R^3 \setminus \Omega}=[\sigma_{r,s}]_{\R^3 \setminus \Omega}$ for every $r \in P$ and for every $s \in \{1,\ldots,g_r\}$. This proves the first part of $(2)$. The second part of $(2)$ now follows immediately from \eqref{eq:basis-outside}.

$(3)$ The existence of the homological Seifert surfaces $S_{r,s}$ is equivalent to $(1)$. Point $(3)$ is a direct consequence of the second part of $(2)$ and of the Poincar\'e-Lefschetz duality theorem.
\end{proof}

We conclude this section by computing the coefficients $\alpha^{r,s}_{i,j}$.
To do it we need to recall the notion of linking number and some properties that will be usefull in the sequel. See, e.g., Rolfsen~\cite[pp.\ 132--136]{ROL76}, Seifert and Threlfall~\cite[Sects.\ 70, 73, 77]{ST80}. The linking number is an integer that, given two  $1$-cycles $\gamma$ and $\eta$ of $\R^3$ with disjoint supports; namely, $|\gamma| \cap |\eta|=\emptyset$, represents the number of times that each curve winds around the other. A possible geometric way to give a rigorous definition is as follows.
Choose a homological Seifert surface $S_\eta=\sum_{q=1}^kb_qf_q$ of $\eta$ in $\R^3$. It is well-known (and easy to see) that there exists a $1$-cycle $\widehat{\gamma}=\sum_{p=1}^h\widehat{a}_p\widehat{e}_p$ homologous to $\gamma$ in $\R^3 \setminus |\eta|$ (and ``arbitrarily close to $\gamma$'' if necessary), which is transverse to $S_\eta$ in the following sense: for every $p \in \{1,\ldots,h\}$ and for every $q \in \{1,\ldots,k\}$, the intersection $|\widehat{e}_p| \cap |f_q|$ is either empty or consists of a single point, which does not belong to $|\partial_1 \widehat{e}_p| \cup |\partial_2 f_q|$.

For every $p \in \{1,\ldots,h\}$ and for every $q \in \{1,\ldots,k\}$, define $L_{pq}:=0$ if $|\widehat{e}_p| \cap |f_q|=\emptyset$ and $L_{pq}:=\sign(\bs{\tau}(\widehat{e}_p) \cdot \bs{\nu}(f_q))$ otherwise. The linking number $\lk(\gamma,\eta)$ between $\gamma$ and $\eta$ is the integer defined as follows:
\begin{equation} \label{eq:def-lk}
\lk(\gamma,\eta):=\sum_{p=1}^h\sum_{q=1}^k\widehat{a}_pb_qL_{pq}.
\end{equation}
This definition is well-posed: it depends only on $\gamma$ and $\eta$, not on the choice of $S_\eta$ and of $\widehat{\gamma}$. %The reader observes that the preceding construction fully justifies the usual heuristic description of the linking number between $\gamma$ and $\eta$ as the number of times that $\gamma$ winds around~$\eta$.

The linking number is symmetric $\lk(\gamma,\eta)=\lk(\eta,\gamma)$, and bilinear
$\lk(a\gamma,\eta)=a \, \lk(\gamma,\eta) \; \text{ for every }a \in \Z$
and,  if $\gamma^* \in Z_1(\R^3;\Z)$ with $|\gamma^*| \cap |\eta|=\emptyset$,
$\lk(\gamma+\gamma^*,\eta)=\lk(\gamma,\eta)+\lk(\gamma^*,\eta) $.

The linking number is a homological invariant in the following sense: if a $1$-cycle $\gamma^*$ of $\R^3$ is homologous to $\gamma$ in $\R^3 \setminus |\eta|$, then
\begin{equation} \label{eq:homol-inv}
\lk(\gamma,\eta)=\lk(\gamma^*,\eta).
\end{equation}
In particular, we have:
\begin{equation} \label{eq:bounds}
\text{$\lk(\gamma,\eta)=0$ if $\gamma$ bounds in $\R^3 \setminus |\eta|$.}
\end{equation}

The linking number can be used to recognize $1$-boundaries of $\T$ among $1$-cycles of $\T$. This is possible by the Alexander duality theorem. Indeed, such a theorem ensures that $H_1(\R^3 \setminus \overline{\Omega};\Z)$ is isomorphic to $H_1(\overline{\Omega};\Z)$, and hence to $\Z^g$ if $g$ is the first Betti number of $\overline{\Omega}$. Furthermore, if $\sigma^*_1,\ldots,\sigma^*_g$ are $1$-cycles of $\R^3$ with support in $\R^3 \setminus \overline{\Omega}$ whose homology classes in $\R^3 \setminus \overline{\Omega}$ form a basis of $H_1(\R^3 \setminus \overline{\Omega};\Z)$, then it holds:
\begin{equation*} %\label{eq:alexander}
\text{a $1$-cycle $\sigma$ of $\T$ is a $1$-boundary of $\T$ if and only if $\lk(\sigma,\sigma^*_i)=0$ for every $i \in \{1,\ldots,g\}$.}
\end{equation*}
For this topic, we refer the reader to \cite{CdeTG02} and to the references mentioned therein.

The linking number can be computed  via a double integral:
\begin{equation}
    \lk (\gamma,\eta) =
    \frac{1}{4 \pi} \oint_\gamma \left( \oint_{\eta}
    \frac{{\bf y}-{\bf x}}{|{\bf y}-{\bf x}|^3}
    \times {\rm d}{\bf s}({\bf y})\right) \cdot {\rm d}{\bf s}({\bf x})\,.
  \end{equation}
For an efficient computation of the linking number see e.g. \cite{Arai13}.

\vspace{.5em}

%\noindent \textbf{Recognition of 1-boundaries.}

For the computation of the coefficients  $\alpha^{r,s}_{i,j}$ we will use also the fact that since
$\partial\Omega$ has a collar in $\R^3 \setminus \Omega$, there exist $1$-cycles $\big\{\widehat{\sigma}_{0,s}^-\big\}_{s=0}^{g_0} \cup \big\{\sigma_{1,s}^-\big\}_{s=1}^{g_1} \cup \ldots \cup \big\{\sigma_{p,s}^-\big\}_{s=1}^{g_p}$ of $\R^3$ with support contained in $\R^3 \setminus \overline{\Omega}$ (obtained by slightly retracting the $1$-cycles $\big\{\widehat{\sigma}_{0,s}\big\}_{s=0}^{g_0} \cup \big\{\sigma_{1,s}\big\}_{s=1}^{g_1} \cup \ldots \cup \big\{\sigma_{p,s}\big\}_{s=1}^{g_p}$ of $\T_\partial$ inside $\R^3 \setminus \overline{\Omega} \;$) such that $[\widehat{\sigma}^-_{0,s}]_{\R^3 \setminus \Omega}=[\widehat{\sigma}_{0,s}]_{\R^3 \setminus \Omega} \;$ for every $s \in \{1,\ldots,g_0\}$ and $[\sigma_{r,s}^-]_{\R^3 \setminus \Omega}=[\sigma_{r,s}]_{\R^3 \setminus \Omega}$ for every $r \in P$ and for every $s \in \{1,\ldots,g_r\}$. In particular, thanks to \eqref{eq:basis-outside-0} and \eqref{eq:basis-outside-r}, we infer that
\begin{equation} \label{eq:basis-outside-0-bis}
\text{$\bigcup_{i \in P}\big\{[\sigma_{i,s}^-]_{\R^3 \setminus \overline{\Omega_0}}\big\}_{s=1}^{g_i} \;$ is a basis of $H_1(\R^3 \setminus \overline{\Omega_0};\Z)$}
\end{equation}
and
\begin{equation} \label{eq:basis-outside-r-bis}
\text{$\big\{[\widehat{\sigma}_{0,s}^-]_{\R^3 \setminus \overline{\Omega_r}}\big\}_{s=1}^{g_0} \cup \bigcup_{i \in P_r }\big\{[\sigma_{i,s}^-]_{\R^3 \setminus \overline{\Omega_r}}\big\}_{s=1}^{g_i} \;$ is a basis of $H_1(\R^3 \setminus \overline{\Omega_r};\Z)$}
\end{equation}
for every $r \in P$.

%\noindent \textbf{Linking number.}

\vspace{.3cm}

For every $k,i \in \{0,1,\ldots,p\}$, define the $(g_k \times g_i)$-matrix $A_{k,i}$ as follows:
\begin{align*}
&A_{0,0}:= \big(\lk(\widehat{\sigma}_{0,l}^-,\sigma_{0,j})\big)_{l,j} \in \Z^{g_0 \times g_0},\\
&A_{0,i}:=\big(\lk(\widehat{\sigma}_{0,l}^-,\widehat{\sigma}_{i,j})\big)_{l,j} \in \Z^{g_0 \times g_i} \; \text{ if $i \in P$,}\\
&A_{k,0}:=\big(\lk(\sigma_{k,l}^-,\sigma_{0,j})\big)_{l,j} \in \Z^{g_k \times g_0} \; \text{ if $k \in P$,}\\
&A_{k,i}:=\big(\lk(\sigma_{k,l}^-,\widehat{\sigma}_{i,j})\big)_{l,j} \in \Z^{g_k \times g_i} \; \text{ if $k,i \in P$ and $k \ne i$,} \\
&A_{k,k}:=\big(\lk(\sigma_{k,l}^-,\widehat{\sigma}_{k,j})\big)_{l,j} \in \Z^{g_k \times g_i} \; \text{ if $k \in P$.}
\end{align*}

\begin{lemma}
For every $k,i \in \{1,\dots, p \}$ the matrices $A_{0,i}$ and $A_{k,0}$ are equal to zero, and if $k \ne i$  then also the matrix $A_{k,i}$ is equal to zero.
%$A_{k,i}:=\big(\lk(\sigma_{k,l}^-,\widehat{\sigma}_{i,j})\big)_{l,j} =\big(\lk(\sigma_{k,l},\widehat{\sigma}_{i,j})\big)_{l,j} =0$ if
\end{lemma}

\begin{proof}
First we notice that for any $l \in \{1,\dots,g_k\}$ the support of the 1-cycle $\sigma_{k,l}$ is contained in $\Gamma_k$ while
for any $j \in \{1,\dots,g_l\}$ the support of the 1-cycle $\widehat{\sigma}_{i,j}$ is contained in $\Gamma_l$. Hence, if $k \ne l$ then  the 1-cycles $\sigma_{k,l}$ and $\widehat{\sigma}_{i,j}$ are disjoint and $\lk(\sigma_{k,l},\widehat{\sigma}_{i,j})=\lk(\widehat\sigma_{i,j},{\sigma}_{k,l})$ is well defined.  Moreover $\lk(\widehat{\sigma}_{0,l}^-,\sigma_{0,j}) = \lk(\widehat{\sigma}_{0,l},\sigma_{0,j})$ and
$\lk(\sigma_{k,l}^-,\widehat{\sigma}_{i,j})=\lk(\sigma_{k,l},\widehat{\sigma}_{i,j})$.

Now it is not difficult to see that
$A_{0,i}=0$ if $i \in P$ because for any $j \in \{1,\dots,g_i \}$ we have $\widehat{\sigma}_{i,j} = \partial_2 S_{i,j} \subset \overline{D_i}$ while for  any $l \in \{1,\dots,g_0\}$, $|\widehat{\sigma}_{0,l}| \subset \Gamma_0$. Since $\Gamma_0 \cap \overline{D_i} = \emptyset$ if $i \in P$, then $ A_{0,i}=\lk(\widehat{\sigma}_{0,l},\sigma_{i,j})=0$.
Analogously
$A_{k,0} =0$ if $k \in P$ because for any $j \in \{1,\dots,g_0\}$, $\sigma_{0,j} = \partial_2 S_{0,j} \subset \R^3 \setminus D_0$ and for any $l \in \{1,\dots,g_l\}$, $|\sigma_{k,l}| \subset \Gamma_k$. Again we have $\Gamma_k \cap \R^3 \setminus D_0 = \emptyset$ if $k \in P$ and then $A_{k,0}=\lk(\sigma_{k,l}^-,\sigma_{0,j})=0$.
Finally
$A_{k,i}=0$ if $k,i \in P$ and $k \ne i$ because for any $j \in \{1,\dots, g_i\}$,  $ \widehat{\sigma}_{i,j} = \partial_2 S_{i,j} \subset \overline{D_i}$, for any $l \in \{1,\dots,g_k\}$, $|\sigma_{k,l}| \subset \Gamma_k$ and $\Gamma_k \cap \overline{D_i}= \emptyset$ if $k \ne i$.

\end{proof}

%(When computing the linking number between two cycles supported  on different conneted component of the boundary it is not necessary to use the rectraction because they are disjoint.)

\subsubsection*{Computation of the coefficients $(\alpha^{0,s}_{i,j})_{i,j}$ for $s \in \{1,\ldots,g_0\}$}

%\textit{Choose $s \in \{1,\ldots,g_0\}$ and compute the coefficients $(\alpha^{0,s}_{i,j})_{i,j}$.}%{i \in P,j \in \{1,\ldots,g_i\}}$.}

Let $G_0:=\sum_{i \in P}g_i=g-g_0$ and let $A_{(0)}$ be the diagonal block matrix with blocks $(A_{k,k})_{k \in P} \in \Z^{G_0 \times G_0}$. It is important to observe that the entries of $A_{(0)}$ are the linking numbers between the representatives of a basis of $H_1(\R^3 \setminus \overline{\Omega_0};\Z)$ (see \eqref{eq:basis-outside-0-bis}) and the representatives of a basis of $H_1(\overline{\Omega_0};\Z)$ (see \eqref{eq:basis-inside-0}). In this way, the Alexander duality theorem applied to $\overline{\Omega}$ ensures that
\begin{equation} \label{eq:det0s}
\big|\det\big(A_{(0)}\big)\big|=1.
\end{equation}

Define the row vectors $\alpha^{0,s}_i:=(\alpha^{0,s}_{i,1},\ldots,\alpha^{0,s}_{i,g_i})$ and
%$\beta^{0,s}_i:=\big(\lk(\sigma_{i,1}^-,\widehat{\sigma}_{0,s}),\ldots,\lk(\sigma_{i,g_i}^-,\widehat{\sigma}_{0,s})\big)$ 
$\beta^{0,s}_i:=\big(\lk(\sigma_{i,1},\widehat{\sigma}_{0,s}),\ldots,\lk(\sigma_{i,g_i},\widehat{\sigma}_{0,s})\big)$ for every $i \in P$, and the column vectors
\[
\alpha^{0,s}:=(\alpha^{0,s}_1,\ldots,\alpha^{0,s}_p)^T \in \Z^{G_0}
\quad \mbox{and} \quad
\beta^{0,s}:=(\beta^{0,s}_1,\ldots,\beta^{0,s}_p)^T \in \Z^{G_0},
\]
where the superscript ``$\, ^T \,$'' denotes the transpose operation.

Bearing in mind the linearity of linking number and its homological invariance, equation \eqref{eq:0s} implies that
\begin{equation} \label{eq:0s-}
%\text{$\lk(\sigma_{k,h}^-,\widehat{\sigma}_{0,s})=\sum_{i \in P}\sum_{j=1}^{g_i}\alpha^{0,s}_{i,j} \lk(\sigma_{k,h}^-,\widehat{\sigma}_{i,j}) \;$ if $k \in P$ and $h \in \{1,\ldots,g_k\}$.}
\text{$\lk(\sigma_{k,h},\widehat{\sigma}_{0,s})=\sum_{i \in P}\sum_{j=1}^{g_i}\alpha^{0,s}_{i,j} \lk(\sigma_{k,h}^-,\widehat{\sigma}_{i,j}) \;$ if $k \in P$ and $h \in \{1,\ldots,g_k\}$.}
\end{equation}
Linear system \eqref{eq:0s-} in the unknowns $(\alpha^{0,s}_{i,j})_{i,j}$ can be rewritten in the following compact form:
\begin{equation} \label{eq:0s-compact}
A_{(0)}\alpha^{0,s}=\beta^{0,s},
\end{equation}
where $\alpha^{0,s}$ is the unknown. Thanks to \eqref{eq:det0s}, equation \eqref{eq:0s} is equivalent to \eqref{eq:0s-compact}.

In this way, we conclude that the coefficients $(\alpha^{0,s}_{i,j})_{i,j}$ can be computed by solving linear system \eqref{eq:0s-compact}, namely solving $p$ linear systems each one of dimension $g_r$, $r=1,\dots,p$.

\subsubsection*{Computation of the coefficients $(\alpha^{r,s}_{i,j})_{i,j}$ for $r \in P$ and $s\in\{1,\dots,g_r\}$.}

%\textit{Choose $r \in P$ and $s \in \{1,\ldots,g_r\}$. It remains to compute the coefficients $(\alpha^{r,s}_{i,j})_{i,j}$.} %$(\alpha^{r,s}_{0,j})_j$ %{j \in \{1,\ldots,g_0\}}$ and $(\alpha^{r,s}_{i,j})_{i,j}$.}% \in P_r,j \in \{1,\ldots,g_i\}}$.} Let us proceed as above.

Given $k \in P$, we define the integer $k_r \in \{0,1,\ldots,p\} \setminus \{r\}$ by setting $k_r:=k-1$ if $k \leq r$ and $k_r:=k$ if $k>r$. Let $G_r:=\sum_{i \in P_r}g_i=g-g_r$ and let $A_{(r)}$ be the diagonal block matrix $(A_{k_r,i_r})_{k,i \in P} \in \Z^{G_r \times G_r}$. By applying the Alexander duality theorem to $\overline{\Omega_r}$ (see \eqref{eq:basis-outside-r-bis} and \eqref{eq:basis-inside-r}), we obtain:
\begin{equation} \label{eq:detrs}
\big|\det\big(A_{(r)}\big)\big|=1.
\end{equation}

Define the row vectors $\alpha^{r,s}_0:=(\alpha^{r,s}_{0,1},\ldots,\alpha^{r,s}_{0,g_0})$, %$\beta^{r,s}_0:=\big(\lk(\widehat{\sigma}_{0,1}^-,\sigma_{r,s}),\ldots,\lk(\widehat{\sigma}_{0,g_0}^-,\sigma_{r,s}\big)$ 
$\beta^{r,s}_0:=\big(\lk(\widehat{\sigma}_{0,1},\sigma_{r,s}),\ldots,\lk(\widehat{\sigma}_{0,g_0},\sigma_{r,s}\big)$
and, for every $i \in P_r$, $\alpha^{r,s}_i:=(\alpha^{r,s}_{i,1},\ldots,\alpha^{r,s}_{i,g_i})$ and %$\beta^{r,s}_i:=\big(\lk(\sigma_{i,1}^-,\sigma_{r,s}),\ldots,\lk(\sigma_{i,g_i}^-,\sigma_{r,s})\big)$. 
$\beta^{r,s}_i:=\big(\lk(\sigma_{i,1},\sigma_{r,s}),\ldots,\lk(\sigma_{i,g_i},\sigma_{r,s})\big)$.
Define also the column vectors
\[
\alpha^{r,s}:=(\alpha^{r,s}_0,\alpha^{r,s}_1,\ldots,\alpha^{r,s}_{r-1},\alpha^{r,s}_{r+1},\ldots,\alpha^{r,s}_p)^T \in \Z^{G_r}
\]
and
\[
\beta^{r,s}:=(\beta^{r,s}_0,\beta^{r,s}_1,\ldots,\beta^{r,s}_{r-1},\beta^{r,s}_{r+1},\ldots,\beta^{r,s}_p)^T \in \Z^{G_r}
\]

By using equation \eqref{eq:rs} and the linking number, we infer that
\begin{equation} \label{eq:rs0-}
%\lk(\widehat{\sigma}_{0,h}^-,\sigma_{r,s})=\sum_{j=1}^{g_0}\alpha^{r,s}_{0,j}\lk(\widehat{\sigma}_{0,h}^-,\sigma_{0,j})+\sum_{i \in P_r}\sum_{j=1}^{g_i}\alpha^{r,s}_{i,j}\lk(\widehat{\sigma}_{0,h}^-,\widehat{\sigma}_{i,j})
\lk(\widehat{\sigma}_{0,h},\sigma_{r,s})=\sum_{j=1}^{g_0}\alpha^{r,s}_{0,j}\lk(\widehat{\sigma}_{0,h}^-,\sigma_{0,j})+\sum_{i \in P_r}\sum_{j=1}^{g_i}\alpha^{r,s}_{i,j}\lk(\widehat{\sigma}_{0,h}^-,\widehat{\sigma}_{i,j})
\end{equation}
if $h \in \{1,\ldots,g_0\}$ and
\begin{equation} \label{eq:rs-}
%\lk(\sigma_{k,h}^-,\sigma_{r,s})=\sum_{j=1}^{g_0}\alpha^{r,s}_{0,j}\lk(\sigma_{k,h}^-,\sigma_{0,j})+\sum_{i \in P_r}\sum_{j=1}^{g_i}\alpha^{r,s}_{i,j}\lk(\sigma_{k,h}^-,\widehat{\sigma}_{i,j})
\lk(\sigma_{k,h},\sigma_{r,s})=\sum_{j=1}^{g_0}\alpha^{r,s}_{0,j}\lk(\sigma_{k,h}^-,\sigma_{0,j})+\sum_{i \in P_r}\sum_{j=1}^{g_i}\alpha^{r,s}_{i,j}\lk(\sigma_{k,h}^-,\widehat{\sigma}_{i,j})
\end{equation}
if $k \in P_r$ and $h \in \{1,\ldots,g_k\}$. Equations \eqref{eq:rs0-} and \eqref{eq:rs-} can be rewritten as follows:
\begin{equation} \label{eq:rs-compact}
A_{(r)}\alpha^{r,s}=\beta^{r,s}.
\end{equation}
Also in this case, for each $r \in P$ matrix $A_{(r)}$ is block diagonal. Thanks to \eqref{eq:detrs}, equation \eqref{eq:rs} and linear system \eqref{eq:rs-compact} are equivalent. Once again, we conclude that the coefficients $(\alpha^{r,s}_{i,j})_{i,j}$ can be computed by resolving linear system \eqref{eq:rs-compact}.

\section{Homological issues for implementation}

Given two different points ${\bf a},{\bf b}$ in $\R^3$, we denote by $[{\bf a}, {\bf b}]$ the oriented segment of $\R^3$ from ${\bf a}$ to ${\bf b}$. The segment of $\R^3$ of vertices $\mb{a}$, $\mb{b}$ is called support of $[\mb{a},\mb{b}]$ and it is denoted by $|[\mb{a},\mb{b}]|$. The unit tangent vector $\bs{\tau}([\mb{a},\mb{b}])$ of the oriented segment $[{\bf a}, {\bf b}]$ is given by  $\bs{\tau}([{\bf a}, {\bf b}]):=\frac{{\bf b} - {\bf a}}{| {\bf b} - {\bf a} |}$. The barycenter of $e=[{\bf a}, {\bf b}]$ is the point of $\R^3$, $B(e)=({\bf a} + {\bf b})/2$.
A (piecewise linear) $1$-chain of $\R^3$ is a finite formal linear combination $\sum_{i=1}^m \alpha_i e_i$ of oriented segments $e_i=[{\bf a}_i, {\bf b}_i]$ of $\R^3$ with integer coefficients $\alpha_i$. We denote by $C_1(\R^3, \Z)$ the abelian group of $1$-chains in $\R^3$.

Analogously, if ${\bf a}$, ${\bf b}$, ${\bf c}$ are three different not aligned points in $\R^3$, we denote by $[{\bf a}, {\bf b}, {\bf c}]$ the oriented triangle of $\R^3$. The triangle of $\R^3$ of vertices $\mb{a},\mb{b},\mb{c}$ is called support of $[\mb{a},\mb{b},\mb{c}]$ and it is denoted by  $|[\mb{a},\mb{b},\mb{c}]|$. The unit normal vector $\bs{\nu}([\mb{a},\mb{b},\mb{c}])$ of the oriented triangle $[{\bf a}, {\bf b}, {\bf c}]$ is obtained by the right hand rule: $\bs{\nu}([{\bf a}, {\bf b}, {\bf c}]):=\frac{ ({\bf b} - {\bf a}) \times ({\bf c} - {\bf a})}{| ({\bf b} - {\bf a}) \times ({\bf c} - {\bf a})|}$. The  barycenter of $f=[{\bf a}, {\bf b}, {\bf c}]$ is the point of $\R^3$, $B(f)=({\bf a} + {\bf b} + {\bf c})/3$. A (piecewise linear) 2-chain of $\R^3$ is a finite formal linear combination $\sum_{i=1}^p \beta_i f_i$ of oriented triangles $f_i=[{\bf a}_i, {\bf b}_i, {\bf c}_i]$ of $\R^3$ with integer coefficients $\beta_i$. We denote by $C_2(\R^3, \Z)$ the abelian group of 2-chains in~$\R^3$.

Finally, if ${\bf a}$, ${\bf b}$, ${\bf c}$, ${\bf d}$ are four different not coplanar points in $\R^3$, we denote by $[{\bf a}, {\bf b}, {\bf c}, {\bf d}]$ the oriented tetrahedron of $\R^3$.
%; namely, the tetrahedron $\{t\mb{a}+s\mb{b}+u\mb{c}+v\mb{d} \in \R^3 \, | \, t,s,u,v \geq 0, t+s+u+v=1\}$ of $\R^3$ of vertices $\mb{a}$, $\mb{b}$, $\mb{c}$, $\mb{d}$, together with the ordering $(\mb{a},\mb{b},\mb{c},\mb{d})$ of its vertices.
The tetrahedron of $\R^3$ of vertices $\mb{a},\mb{b},\mb{c},\mb{d}$ is called support of the oriented tetrahedron $[\mb{a},\mb{b},\mb{c},\mb{d}]$ and it is denoted by $|[\mb{a},\mb{b},\mb{c},\mb{d}]|$. The  barycenter of $t=[{\bf a}, {\bf b}, {\bf c}, {\bf d}]$ is the point of $\R^3$, $B(t)=({\bf a} + {\bf b} + {\bf c}+{\bf d})/4$. A (piecewise linear) 3-chain of $\R^3$ is a finite formal linear combination $\sum_{i=1}^q d_i t_i$ of oriented tetrahedra $t_i=[{\bf a}_i, {\bf b}_i, {\bf c}_i, {\bf d}_i]$ of $\R^3$ with integer coefficients $d_i$. We denote by $C_3(\R^3, \Z)$ the abelian group of 3-chains~in~$\R^3$.

We indicate by $\E$, $\F$ and $\K$ the sets of oriented edges, oriented faces and oriented tetrahedra of $\T$, respectively.

Let us recall the definitions of dual vertices, dual edges and dual faces of $\T$. We equip the dual edges and the dual faces with the natural orientation induced by the right hand rule.
\begin{itemize}
 \item
 For every tetrahedron $t \in K$, the dual vertex $D(t)$ of $\T$ associated with $t$ is defined as the barycenter of $t$:
 $ D(t):=B(t)$.

We denote by $V'$ the set $\{D(t) \in \R^3 \, | \, t \in K\}$ of all dual vertices of $\T$.

\item
For every oriented face $f=[\vv,\mb{w},\mb{y}] \in \F$, the oriented dual edge $D(f)$ of $\T$ associated with $f$ is the element of $C_1(\R^3;\Z)$ defined as follows: if $K(f)$ denotes the set $\big\{t \in K \, \big| \, \{\vv,\mb{w},\mb{y}\} \subset t\big\}$; namely, the set of tetrahedra of $\T$ incident on $f$, we set
\[
 D(f):=\sum_{t \in K(f)} \sign \big(\bs{\nu}(f) \cdot \bs{\tau}([B(f),B(t)]) \big) \, [B(f),B(t)],
\]
where $\sign:\R \setminus \{0\} \lra \{-1,1\}$ denotes the function given by $\sign(s):=-1$ if $s<0$ and $\sign(s):=1$ otherwise.

$D(f)$ can be described as follows. If the (oriented) face $f$ is internal, then $f$ is the common face of two tetrahedra $t_1$ and $t_2$ of $\T$, and the support of $D(f)$ is the union of the segment joining $B(f)$ with $B(t_1)$ and of the segment joining $B(f)$ and $B(t_2)$. If $f$ is a boundary face, then $f$ is face of just one tetrahedron $t$, and the support of $D(f)$ is the segment joining $B(f)$ with $B(t)$. In both cases, $D(f)$ is endowed with the orientation induced by $f$ via the right hand rule.

We denote by $\E'$ the set $\{D(f) \in C_1(\R^3;\Z) \, | \, f \in \F\}$ of all oriented dual edges of $\T$. % Moreover, we call (non-oriented) dual edge of $\T$ a $2$-subset $\{v',w'\}$ of $\R^3$ such that $\{v',w'\}=|\partial_1e'|$ for some $e' \in \E'$. We indicate by $E'$ the set of all (non-oriented) dual edges of $\T$.

 \item
 For every oriented edge $e=[\vv,\mb{w}] \in \E$, the oriented dual face $D(e)$ of $\T$ associated with $e$ is the element of $C_2(\R^3;\Z)$ defined as follows: if $F(e)$ denotes the set $\big\{f \in F \, \big| \, \{\vv,\mb{w}\} \subset f\big\}$; namely, the set of  faces of $\T$ incident on $e$, then we set
\[
D(e):=\sum_{f \in F(e)} \sum_{t \in K(f)} \sign \big( \bs{\tau}(e) \cdot \bs{\nu}([B(e),B(f),B(t)]) \big) \,   [B(e),B(f),B(t)]\, .
\]
The reader observes that the support of $D(e)$ is the union of triangles of $\R^3$ with vertices  $B(e)$, $B(f)$, and $B(t)$, where $f$ varies in $F(e)$ and $t$ in $K(f)$. Such triangles are oriented by $e$ via the right hand rule.

We denote by $\F'$ the set $\{D(e) \in C_2(\R^3;\Z) \, | \, e \in \E\}$ of all oriented dual faces of $\T$.
\end{itemize}

The preceding three definitions determine the bijection $D:K \cup \F \cup \E \lra V' \cup \E' \cup \F'$ such that $D(K)=V'$, $D(\F)=\E'$ and $D(\E)=\F'$.

We need also to describe the closed block dual barycentric complex of the triangulation $\T_\partial$ of $\partial\Omega$ induced by $\T$. Recall that  $V_\partial$, $\E_\partial$ and $\F_\partial$ denote the sets of vertices, of oriented edges and of oriented faces of $\T_\partial$, respectively.

\begin{itemize}
 \item
 For every oriented face $f \in \F_\partial$, the dual vertex $D_\partial(f)$ of $\T_\partial$ associated with $f$ is defined as the barycenter of $f$: $ D_\partial(f):=B(f)$.

We denote by $V'_\partial$ the set $\{D_\partial(f) \in \R^3 \, | \, f \in \F_\partial\}$ of all dual vertices of $\T_\partial$.

 \item
For every oriented edge $e \in \E_\partial$, the oriented dual edge $D_\partial(e)$ of $\T_\partial$ associa\-ted with $e$ is the element of $C_1(\R^3;\Z)$ defined as follows. Let $f_1$ and $f_2$ be the oriented faces in $\F_\partial$ incident on $e$, and let $\mb{n}(f_1)$ and $\mb{n}(f_2)$ be the outward unit normals of $\partial\Omega$ at $B(f_1)$ and at $B(f_2)$, respectively. Then we set
\[
D_\partial(e):=\sum_{i=1}^2 \sign\big(\bs{\tau}(e) \cdot (\mb{n}(f_i) \times \bs{\tau}([B(e),B(f_i)])) \big)[B(e),B(f_i)].
\]
$D_\partial(e)$ can be described as follows. By interchanging $f_1$ with $f_2$ if necessary, we can suppose that $f_1$ is on the left of $e$ and $f_2$ on the right of $e$ with respect to the orientation of $\partial\Omega$ induced by its outward unit vector field. Then we have:
\[
D_\partial(e)=[B(f_1),B(e)]+[B(e),B(f_2)]\, .
\]
We denote by $\E'_\partial$ the set $\{D_\partial(e) \in C_1(\R^3;\Z) \, | \, e \in \E_\partial\}$; namely, the set of all oriented dual edges of $\T_\partial$. %Moreover, we call (non-oriented) dual edge of $\T_\partial$ a $2$-subset $\{\mb{v}',\mb{w}'\}$ of $V'_\partial$ such that $\{\mb{v}',\mb{w}'\}=|\partial_1 e'|$ for some $e' \in \E'_\partial$. We indicate by $E'_\partial$ the set of all (non-oriented) dual edges of $\T_\partial$.

\item
For every ${\bf v} \in V_\partial$ , the oriented dual face $D_\partial ({\bf v})$ of $\T_\partial$ associa\-ted with ${\bf v}$ is the element of  $C_2(\R^3; \Z)$ defined as follows. If $E_\partial({\bf v})$ denotes the set $\{ e \in E_\partial \, | \, \{ {\bf v}\} \subset e\}$; namely the set of edges of $\T_\partial$ incident on ${\bf v}$ and,  for any edge $e \in  E_\partial$, $ F_\partial(e)$ denotes the set of oriented faces in $\partial \Omega$ incident in $e$ then
$$
D_\partial({\bf v})= \sum_{e \in E_\partial ({\bf v})} \sum_{f \in F_\partial (e)}\sign \left({\bf n}(f) \cdot \nu([{\bf v},B(e),B(f)]) \right) [{\bf v},B(e),B(f)]\, ,
$$
being ${\bf n}(f)$ the outward unit normal vector of $\partial \Omega$ at $B(f)$.

We denote by $\F'_\partial$ the set $\{D_\partial({\bf v}) \in \R^3 \, | \, {\bf v} \in V_\partial\}$ of all dual faces of $\T_\partial$.

\end{itemize}

\subsection{Construction of the retraction}

>From the computational point of view, in order to construct $g$ 1-boundaries with supports contained in $\partial\Omega$ whose homology classes in $\R^3 \setminus \Omega$ form a basis of $H_1(\R^3 \setminus \Omega;\Z)$ is more convenient to construct
$1$-cycles $\big\{{\sigma}_{0,s}^+\big\}_{s=0}^{g_0} \cup \big\{\widehat\sigma_{1,s}^+\big\}_{s=1}^{g_1} \cup \ldots \cup \big\{\widehat\sigma_{p,s}^+\big\}_{s=1}^{g_p}$ of $\R^3$ with support contained in $\Omega$ that are
a retraction of the cycles $\big\{\sigma_{0,s}\big\}_{s=0}^{g_0} \cup \big\{\widehat \sigma_{1,s}\big\}_{s=1}^{g_1} \cup \ldots \cup \big\{\widehat\sigma_{p,s}\big\}_{s=1}^{g_p}$ of $\T_\partial$
such that  $\lk(\widehat{\sigma}_{0,l}^-,\sigma_{0,j})= \lk(\widehat{\sigma}_{0,l},\sigma_{0,j}^+)$ and $\lk(\sigma_{k,l}^-,\widehat{\sigma}_{k,j})=\lk(\sigma_{k,l},\widehat{\sigma}_{k,j}^+)$. 

Let us see how to compute such a retraction. We can consider just the case of a simple loop $\eta$ with $|\eta| \subset \partial \Omega$.
%Then for each ${\bf v} \in |\eta|$ we consider a different 2-chain $\hbox{fun}({\bf v}, \eta)$ whose support is a subset of $|D_\partial({\bf v})|$ that, roughly speaking is the part of $D_\partial({\bf v})$ on the left of $\eta$.
%that is an element of $C_2(\R^3; \Z)$ defined as follows: if $E_\partial(v)$ denotes the set $\{ e \in E_\delta \, | \, v \subset e\}$; namely the set of edges of $\T_\partial$ incident on $v$ and, for any edge $e \in E_\delta$, $F_\partial(e)$ denotes the set of oriented faces in $\partial \Omega$ incident in $e$ then
%$$
%D_\partial(v)= \sum_{e \in E_\partial (v)} \sum_{f \in F_\partial (e)}\sign \left({\bf n} \cdot \nu([v,B(e),B(f)]) \right) [v,B(e),B(f)]\, .
%$$
%Being $\eta$ a simple loop,
Then for each ${\bf v} \in |\eta|$ there exist exactly two oriented edges $e_p=[{\bf v}_p,{\bf v}]$ and $e_s=[{\bf v},{\bf v}_s]$ such that the coefficents of $e_p$ and $e_s$ in $\eta$ are both equal one. % positive coe $\tau(\eta) = \tau ([{\bf v}_p,{\bf v}]) $, and  $\{{\bf v}, {\bf v}_s\} \subset | \eta|$ and $\tau(\eta)= \tau ([{\bf v},{\bf v}_s])$.

For each vertex ${\bf v} \in V_\partial$, $F_\partial({\bf v})$ denotes the set of oriented faces in $\partial \Omega$ incident in ${\bf v}$. Then if ${\bf v} \in |\eta|$  we denote $\hbox{left}({\bf v}, \eta)$ the faces $f \in F_\partial({\bf v})$  that are on the left with respect to $\eta$. More precisely, denoting by $V({\bf v})= \{ {\bf w} \in V_\partial \, | \, |[{\bf w}, {\bf v}]| \in E_\partial \}$ we sort the vertices in $V({\bf v})$ in the following way: we set ${\bf w}_0= {\bf v}_p$ and for $m>0$,
 ${\bf w}_{m}$ is the unique element of  $V({\bf v})$ such that $\nu([{\bf w}_{m-1}, {\bf v},{\bf w}_m])$ coincides with the outward unit normal of $\partial \Omega$ at these face. Clearly there exists $m^* \ge 1$ such that ${\bf w}_{m^*} = {\bf v}_s$. We define
$$
\hbox{left}({\bf v}, \eta):= \{ f \in F_\partial \, | \,
|f|=|[{\bf w}_{m-1}, {\bf v}, {\bf w}_m]|  \hbox{ for some } m\in \{1,\dots, m^*\} \}\, .
$$
%$$
%\hbox{left}({\bf v}, \eta):= \{ f \in F_\partial ({\bf v})\, | \, \sign({\bf n}(f) \cdot \nu([B(e_p), B(f), B(e_s)]) < 0 \}\, .
%$$
Then we denote $\hbox{fun}({\bf v}, \eta)$ the 2-chain
$$
\hbox{fun}({\bf v},\eta)= \sum_{e \in  E_\partial ({\bf v})} \sum_{f \in  F_\partial (e) \cap \hbox{left}({\bf v}, \eta)}
\sign({\bf n}(f) \cdot \nu([{\bf v},B(e),B(f)])) [{\bf v},B(e),B(f)]\, ,
$$
namely, the subchain of $D_\partial({\bf v})$ with support on the left of $\eta$.

First we replace $\eta$ with $\hat \eta= \eta -  \partial_2 \left( \sum_{{\bf v} \in |\eta|} \hbox{fun}({\bf v}) \right)$. Notice that since $\partial \Omega$ is orientable (????) then $\hat \eta$ is a formal linear combination of oriented boundary dual edges: $\hat \eta = \sum_{e \in \mathcal E_\partial} c_e D_\partial(e)$.
Then we define the interior retraction $\eta^+$ in the following way: $\eta^+= \hat \eta - \sum_{e \in \mathcal E_\partial} c_e \partial_2 D(e)$. $\eta^+$ is a linear combination of oriented {\em interior} dual edges.

%The reader observes that, for every $e \in \E_\partial$, $\coil(e)-D_\partial(e)$ is a $1$-chain of $\A'$, whose expression as a formal linear combination contains only oriented edges in $\E'$; namely, $\coil(e)-D_\partial(e)=\sum_{e' \in \E' \cup \E'_\partial}a_{e'}e'$ for some (unique) integer $a_{e'}$ such that $a_{e'}=0$ for every $e' \in \E'_\partial$.

%It is not difficult to see that $A_{0,i}=0$ if $i \in P$ because for $j=1,\dots,g_i$ we have $\widehat{\sigma}_{i,j} = \partial_2 S_{i,j} \subset \overline{D_i}$ while for $l=1,\dots,g_0$ we have $|\widehat{\sigma}_{0,l}| \subset \Gamma_0$ and $\Gamma_0 \cap \overline{D_i} = \emptyset$ if $i \in P$. Analogously $A_{k,0} =0$ if $k \in P$ because for $j=1,\dots,g_0$, $\sigma_{0,j} = \partial_2 S_{0,j} \subset \R^3 \setminus D_0$, for $l=1,\dots,g_l$, $|\sigma_{k,l}| \subset \Gamma_k$ and again $\Gamma_k \cap \R^3 \setminus D_0 = \emptyset$ if $k \in P$. Moreover $A_{k,i}=0$ if $k,i \in P$ and $k \ne i$ because for $j=1,\dots, g_i$,  $ \widehat{\sigma}_{i,j} = \partial_2 S_{i,j} \subset \overline{D_i}$, for $l=1,\dots,g_k$, $|\sigma_{k,l}| \subset \Gamma_k$ and $\Gamma_k \cap \overline{D_i}= \emptyset$ if $k \ne i$.

%%%%%%%

\subsection{Construction of homological Seifert surfaces} \label{sec:Seifert}

Given an orientation of the edges and of the faces of the triangulation $\T$ of $\overline{\Omega}$, the problem of computing homological Seifert surfaces  can be formulated as a linear system with as many unknowns as faces and as many equations as edges of $\T$.

Let $\gamma=\sum_{e \in \E}a_ee$ be a given $1$-boundary of $\T$. A $2$-chain $S=\sum_{f \in \F}b_ff$ of $\T$ is a homological Seifert surface of $\gamma$ in $\T$ if its coefficients $\{b_f\}_{f \in \F}$ satisfy the following equation: % in $C_1(\T;\Z)$:
\begin{equation} \label{eq:eq-hss}
\sum_{f \in \F}b_f\partial_2f=\sum_{e \in \E}a_ee, .
\end{equation}
%being $\partial_2: C_2(\T;\Z) \lra C_1(\T;\Z)$ the boundary operator and $C_k(\T;\Z)$ the space of $k$-chains of $\T$.

We can write this equation more explicitly as a linear system. Given $e \in \E$, let $\F(e)$ be the set $\big\{f \in \F \, \big| \, |e| \subset |f|\big\}$ of oriented faces in $\F$ incident on $e$ and let $\o_e:\F(e) \lra \{-1,1\}$ be the function sending $f \in \F(e)$ into the coefficient of $e$ in the expression of $\partial_2f$ as a formal linear combination of oriented edges in $\E$. Equation \eqref{eq:eq-hss} is equivalent to the linear system
$$
\sum_{f \in \F(e)}\o_e(f)b_f=a_e \quad \forall \,  e \in \E\, ,
$$
where the unknowns $\{b_f\}_{f \in \F}$ are integers.

The matrix of this linear system is the incidence matrix between faces and edges of $\T$. Its entries take values in the set $\{-1,0,1\}$. This matrix is very sparse because it has just three nonzero entries per columns and the number of nonzero entries on each row is equal to the number of faces incident on the edge corresponding to the row. This kind of problems are usually solved using the Smith normal form, a computationally demanding algorithm even in the case of sparse matrices (see e.g. \cite{Mun84}, \cite{DSV01}).

A first difficulty to devise a general and efficient algorithm to compute a homological Seifert surface $S$ of a given $1$-boundary $\gamma$ of $\T$ is that the problem has not a unique solution. If $\mk{t}$ is the number of tetrahedra of $\T$ and $\Gamma_0,\Gamma_1,\ldots,\Gamma_p$ are the connected components of $\partial\Omega$, then the kernel of the incidence matrix is a free abelian group of rank $\mk{t}+p \,$; namely, it is isomorphic to $\Z^{\mk{t}+p}$. One of its basis is given by the boundaries of tetrahedra of $\T$ and by the $2$-chains $\gamma_1,\ldots,\gamma_p$ associated with the triangulations of $\Gamma_1,\ldots,\Gamma_p$ induced by $\T$.

A natural strategy to obtain a unique solution $S$ is to add $\mk{t}+p$ equations, by setting equal to zero the unknowns corresponding to suitable faces $f_1,\ldots,f_{\mk{t}+p}$ of $\T$. From the geometric point of view, this is equivalent to impose that the homological Seifert surface $S$ of $\gamma$ does not contain the faces $f_1,\ldots,f_{\mk{t}+p}$. From the computational point of view, it is equivalent to eliminate some unknowns of the problem to obtain an equivalent solvable linear system with a unique solution. We will use graph techniques to describe which coefficients set equal zero. More precisely, we introduce the \textit{complete dual graph of $\T$} denoted by $\A'$.

To do that we need to recall some notions of homology theory (see e.g. \cite{Mun84}).

\begin{defin} \label{def:cdg}
We call $\A':=(V' \cup V'_\partial,E' \cup E'_\partial)$ \emph{complete dual graph of $\T$}.
A \emph{$1$-chain of $\A'$} is a formal linear combination of oriented dual edges in $\E' \cup \E'_\partial$ with integer coefficients.  A $1$-chain $\gamma$ of $\A'$ is called \emph{$1$-cycle of $\A'$} if $\partial_1\gamma=0$. %We denote by $C_1(\A';\Z)$ the abelian subgroup of $C_1(\R^3;\Z)$ consisting of all $1$-chains of $\A'$, and by  $Z_1(\A';\Z)$ the abelian subgroup of $Z_1(\R^3;\Z)$ consisting of all $1$-cycles of $\A'$.
\end{defin}

%More precisely, we introduce the \textit{complete dual graph of $\T$} denoted by $\A'$. Let $F$ be the set of faces of $\T$, $F_\partial$ the set of faces of $\T$ contained in $\partial\Omega$ and $E_\partial$ the set of edges of $\T$ contained in $\partial\Omega$. The dual edge $D(f)$ of a face $f \in F$ and the dual edge $D(e)$ of an edge $e \in E_\partial$ are defined in the following way. If $f \in F_\partial$, then it is contained in a unique tetrahedron $t$ and $D(f)f:=\{B(f),B(t)\}$, where $B(f)$ is the barycenter of $f$ and $B(t)$ the barycenter of $t$. If $f$ is an internal face of $\T$ (namely $f \in F \setminus F_\partial$), then it is the common face of exactly two tetrahedra $t_1$ and $t_2$, and $D(f)f:=\{B(t_1),B(t_2)\}$. Similarly, if $e \in E_\partial$, then it is the common edge of exactly two faces $f_1,f_2$ in $F_\partial$, and $D(f):=\{B(f_1),B(f_2)\}$. The vertices of $\A'$ are the barycenters of tetrahedra of $\T$ and the barycenters of faces in $F_\partial$, and the edges of $\A'$ are the dual edges $\{\epsilon'_f\}_{f \in F}$ and $\{\epsilon'_\ell\}_{\ell \in E_\partial}$.

Our idea is to consider a suitable spanning tree $\B'$ of $\A'$ and to set equal to zero the unknowns corresponding to faces of $\T$ whose dual edge belongs to $\B'$. The total number of arcs in the spanning tree $\B'$ is equal to the number of tetrahedra of $\T$ plus the number of faces of $\T$ contained in $\partial \Omega$ minus one, but, clearly, not all the arcs of $\B'$ correspond to faces of $T$ since there are also arcs corresponding to edges of $\T$ contained in $\partial \Omega$.
The choice of $\B'$ is promising if and only if the number of faces of $\T$ whose dual edge belongs to $\B'$; namely, the number of arcs of $\B'$ not contained in $\partial\Omega$ is equal to $\mk{t}+p$ but not all the spanning trees of $\A'$ satisfy this equality. It is not difficult to see that for all spanning tree $\B'$ of $\A'$, $N_{\B'}\ge \mk{t}+p$. The equality holds true if and only if for each $i \in \{ 0,1,\dots,p \}$ the graph $\B'_i$ induced by $\B'$  on $\Gamma_i$ is a spanning tree of $\A'_i$, the graph induced by $\A'$ on $\Gamma_i$. If the spanning tree $\B'$ of $\A'$ has the latter property, then we call it \textit{Seifert dual spanning tree of $\T$}.

Let $\B'=(V' \cup V'_\partial,N')$ be a Seifert dual spanning tree of $\T$ and let $\mc{N}'$ be its set of oriented dual edges. In \cite{ABGS15} we proved  that the following linear system
\begin{equation} \label{eq:LS}
\left\{
 \begin{array}{ll}
 \sum_{f \in \F(e)}\o_e(f)b_f=a_e & \text{if $e \in \E$} \\
 b_f=0  & \text{if $\epsilon'_f \in \mc{N}'$}
 \end{array}
\right.
\end{equation}
has a unique solution.
In \cite{ABGS15} we give also an explicit formula for the coefficients of the solution of \eqref{eq:oe}.
Roughly speaking the coefficient in $S$ of any face $f$ with $D(f) \in \mc{N}'$ is equal to the linking number between $\gamma$ and the unique $1$-cycle, $\cb(D(f))$ of $\A'$ with all the edges except $D(f)$ contained in $\B'$. But this two cycles could intersect on $\partial \Omega$ and in this case is necessary, in order to define the linking number, to ``retract'' $\gamma$ inside $\overline \Omega$.
More precisely we prove that
\[
b_f=\lk\big(R_+(\gamma),\cb(D(f))\big)
\]
for every $f \in \F$.
The cycle $R_+(\gamma)$ is defined in the following way. For every oriented edge $e=[\vv,\mb{w}]$ in $\E_\partial$, choose a tetrahedron $t_e \in K$ incident on $e$ (namely, $\{\vv,\mb{w}\} \subset t_e$), denote by $\mb{d}_e$ the barycenter of the triangle of $\R^3$ of vertices $\vv$, $\mb{w}$, $B(t_e)$, and define the $1$-chain $r_+(e)$ of $\R^3$ by setting
\[
r_+(e):=[\vv,\mb{d}_e]+[\mb{d}_e,\mb{w}].
\]
Given $\xi=\sum_{e \in \E}\alpha_ee$, we define:
\[
R_+(\xi):=\sum_{e \in \E \setminus \E_\partial} \alpha_e e + \sum_{e \in \E_\partial} \alpha_e r_+(e).
\]

To compute the solution of \eqref{eq:LS} is convenient to adopt an elimination procedure and to use the explicit formula if it is necessary to restart the elimination procedure.

Let us set  $\G=\{f \in \F \, | \, D(f) \in \mc{N}'\}$.

\begin{alg} \label{alg:main}
\begin{enumerate}
\item[]
  \item $\mc R:= \mc{G}$, $\mc D:=\E$.
  \item  while $\mc R \ne \F$
  \begin{enumerate}
    \item $n_{\mc R}:=card(\mc R)$
    \item for every $e \in \mc D$
    \begin{enumerate}
      \item if every oriented face of $\F(e)$ belong to $\mc R$
\begin{enumerate}
\item $\mc D=\mc D\setminus \{e\}$
\end{enumerate}
      \item if exactly one oriented face  $f^* \in \F(e)$ does not belong to $\mc R$
      \begin{enumerate}
        \item compute $b_f$ via \eqref{eq:LS} %$\sum_{f \in \F(e)}\o_e(f)b_f=a_e$
        \item $\mc R=\mc R \cup \{ f \}$
        \item $\mc D=\mc D\setminus \{e\}$
      \end{enumerate}
    \end{enumerate}
    \item if $card(\mc R)=n_{\mc R}$
    \begin{enumerate}
      \item pick $f \not \in \mc R$ and compute $b_f=\lk(R_+(\gamma),\cb(D(f)))$
      \item $\mc R=\mc R \cup \{ f \}$
    \end{enumerate}
  \end{enumerate}
\end{enumerate}
\end{alg}

We have shown in  \cite{ABGS15}  that very often, (2.c) never occours and the homological Seifert surface can be computed by a very fast elimination procedure. In the examples that we tried the elimantion prcedure fails just when considering a non trivial computational domain that is a cube with a knotted cavity, and a boundary that embrace twice the cavity. In this case it was enough to use once tha explicti formula  to restart the eeimination procedure.

Concerning the existence and the construction of internal homological Seifert surfaces of $\gamma$; namely, homological Seifert surfaces of $\gamma$ formed only by internal faces of $\T$ we proved in \cite{ABGS15} that a necessary condition for the solution of an internal homological Seifert surface is that the boundary $\gamma$ must be corner free, namely, no edge of $\gamma$ belongs to two faces on $\partial \Omega$ of the same tetrahedra. Clearly if the mesh is such that no tetrahedra has two faces on $\partial \Omega$ then each boundary is corner free. Moreover in \cite{ABGS15} we identify a family of Seifert dual spanning trees of $\T$, the so called strongly-Seifert dual spanning trees, such that if the boundary $\gamma$ is corner free then the computed homological Seifert surface using such a Seifert dual spanning tree and Algorithm~\ref{alg:main} is internal.

Let us denote by plug (the support of) the dual edge of a boundary face. A maximal plug set is a set of disjoint plugs that is not subset of any other set of disjoint plugs. If the mesh is such that no tetrahedra has two faces on $\partial \Omega$ then the set of all plugs is the unique maximal plug set. Notice that if a tetrahedra has two faces on $\partial \Omega$ then the dual edges of these two faces are not disjoint because the barycenter of the tetrahedra is a common point.

Given a Seifert dual spanning tree $\B'$ of $\T$, we say that $\B'$ is a strongly-Seifert dual spanning tree of $\T$ if it contains a maximal plug-set of $\T$.

\section{Some numerical experiments}
We have implemented the algorithm proposed in this paper in C++. All computations have been run on an Intel Core i7-3720QM @ 2.60GHz laptop with 16Gb of RAM.

The first elementary example is a solid torus with a concentric toric cavity. The boundary of the domain has two connected components and none of them is homologically trivial. The generators of  $H_1(\R^3 \setminus \Omega; \Z)$ are the two cycles $\hat \sigma_1$, $\hat \sigma_2$ represented in Figure~\ref{fig:intro1} as continuous lines. Clearly none of them is the boundary of a 2-chain contained in $\overline \Omega$. Therefore, the first step is to complete each one with a cycle trivial in $H_1(\R^3 \setminus \Omega; \Z)$ in order to obtain a 1-boundary in the same homology class.

 In Figure~\ref{fig:torostorato} we show the two representatives of $H_2(\overline \Omega, \partial \Omega; \Z)$ on the left the one corresponding to to the cycle $\hat{\sigma}_{1}$ and on the right the one corresponding to the cycle $\hat \sigma_{2}$.

\begin{figure}[!h]
\centering
\includegraphics[width=.8\columnwidth]{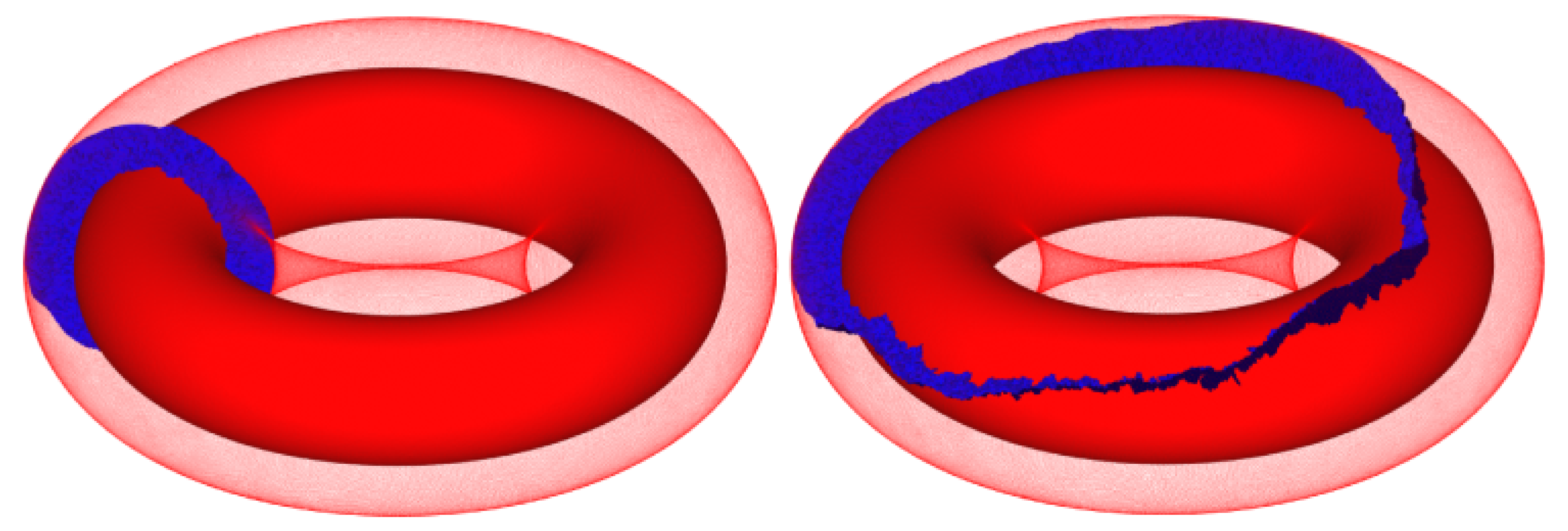}
\caption{The torus with a toric cavity. Representatives of a basis of $H_2(\overline \Omega, \partial \Omega;\Z)$ for the finest mesh are shown.}
\label{fig:torostorato}
\end{figure}

Table \ref{tab:torostorato} contains the details on the number of edges and faces in the complex for four different meshes and the corresponding computational time divided into four contributions: \emph{Mesh pre-processing} represent the time spent for loading the mesh from hard disk and computing all incidences between the elements of the complex. \emph{Hiptmair--Ostrowski} is the time spent for computing the bases of $H_1(\overline \Omega; \Z)$ and $H_1(\R^3 \setminus \Omega;\Z)$ with the algorithm introduced in \cite{HO02}. We remark that each one of the $g$ elements of the constructed bases is supported in a single connected component of the boundary. \emph{Boundary retrieval} is the time employed to find the 1-boundaries from the homology basis, which is the main contribution of this paper. Finally, \emph{elimination algorithm} represents the time needed for the construction of the homological Seifert surfaces with the iterative elimination algorithm introduced in \cite{ABGS15}.

\begin{table}[!h]
\centering
\rowcolors{2}{gray!25}{white}
\begin{tabular}{lcccc}
\rowcolor{gray!50}
Benchmark torus with toric cavity & Mesh 1 & Mesh 2 & Mesh 3 & Mesh 4\\
Edges                   & 51521 & 145963 & 1321902 & 10238231\\
Faces                   & 76330 & 227314 & 2177158 & 17210016\\
Mesh pre-processing [s]            & 0.607 & 1.800  & 17.76  & 141.2 \\
Hiptmair--Ostrowski [s]            & 0.084 & 0.216  & 0.863  & 3.909 \\
Boundary retrieval [s]             & 0.012 & 0.034  & 0.122  & 0.513 \\
Elimination algorithm [s]          & 0.061 & 0.193  & 2.720  & 24.52\\
Total Time (this paper) [s]        & 0.764 & 2.243  & 21.46  & 170.1\\
Total Time (GMSH \cite{gmsh}) [s]  & 1.544 & 5.538  & 86.28  & $>2$ hours\\
Speedup                            & 2.0   & 2.5    & 4.0    & $-$\\
\end{tabular}
 \caption{The torus with a toric cavity: the number of geometric elements of the triangulation and the computational time.}\label{tab:torostorato}
\vspace{-.2cm}
\end{table}

In Table~\ref{tab:torostorato} (and in the next tables) we include also the time spent by a state-of-the-art implementation of the purely algebraic algorithm to compute the $H_2(\overline \Omega, \partial \Omega; \Z)$ generators contained in the popular mesh generator GMSH (see \cite{gmsh}). As one can see, the speed up of the technique proposed in this paper is about 2 in case of small meshes but gets much bigger when considering real-life meshes with millions of tetrahedra. For example, in the last mesh comprising about 10 millions edges, the generators have been computed in less than 3 minutes with the technique introduced in this paper, whereas GMSH did not produced generators after 2 hours of wall time (GMSH would require much more memory to run this test in reasonable time). We also should remark that in the tables and figures we always consider the wall time for our algorithm, whereas GMSH output the CPU time (which, given that does not include time for accessing memory and CPU time, is smaller than wall time). Finally, we also mention that in the GMSH CPU time we have not counted the time needed to load the mesh into GMSH from hard disk and the time required for build the additional data structure for the cell complex incidences.

In the algorithm proposed in this paper the more expensive computation is the one concerning the linking number that, in the worst case, has a computational cost proportional to the square of the number of edges in the boundary of $\Omega$. The total number of linking numbers to be computed is $\sum_{r=0}^P (2g_r)^2$ for the automatic construction of generators of a basis of  $H_1(\R^3 \setminus \Omega;\Z)$ using the algorithm by Hiptmair--Ostrowski, plus the computation of the coefficient $\beta^{r,s}$, $r=0,\dots,P$ that are $ \sum_{r=0}^P g_r \sum_{s=0\, s \ne r}^P g_s = \sum_{r=0}^P g_r  (g-g_r)= g^2 - \sum_{r=0}^P g_r^2$. So the total number is  $g^2+3\sum_{r=0}^P g_r^2$ that, for this first example means, $10$ linking number to be computed.

All observations related to this simple benchmark still hold true for other numerical experiments.
In our second example the domain is the complement of the Borromean rings ($g=3$) with respect to a box. The number of connected components of the boundary is $4$ and the first Betti number of the domain is equal to $3$. The number of linking numbers to be computed is $9+3(1+1+1)= 18$.
In Figure~\ref{fig:borromean2} we show three representatives of a basis of $H_2(\overline \Omega, \partial \Omega,\Z)$  for two different meshes.
\begin{figure}[!h]
\centering
\includegraphics[width=.7\columnwidth]{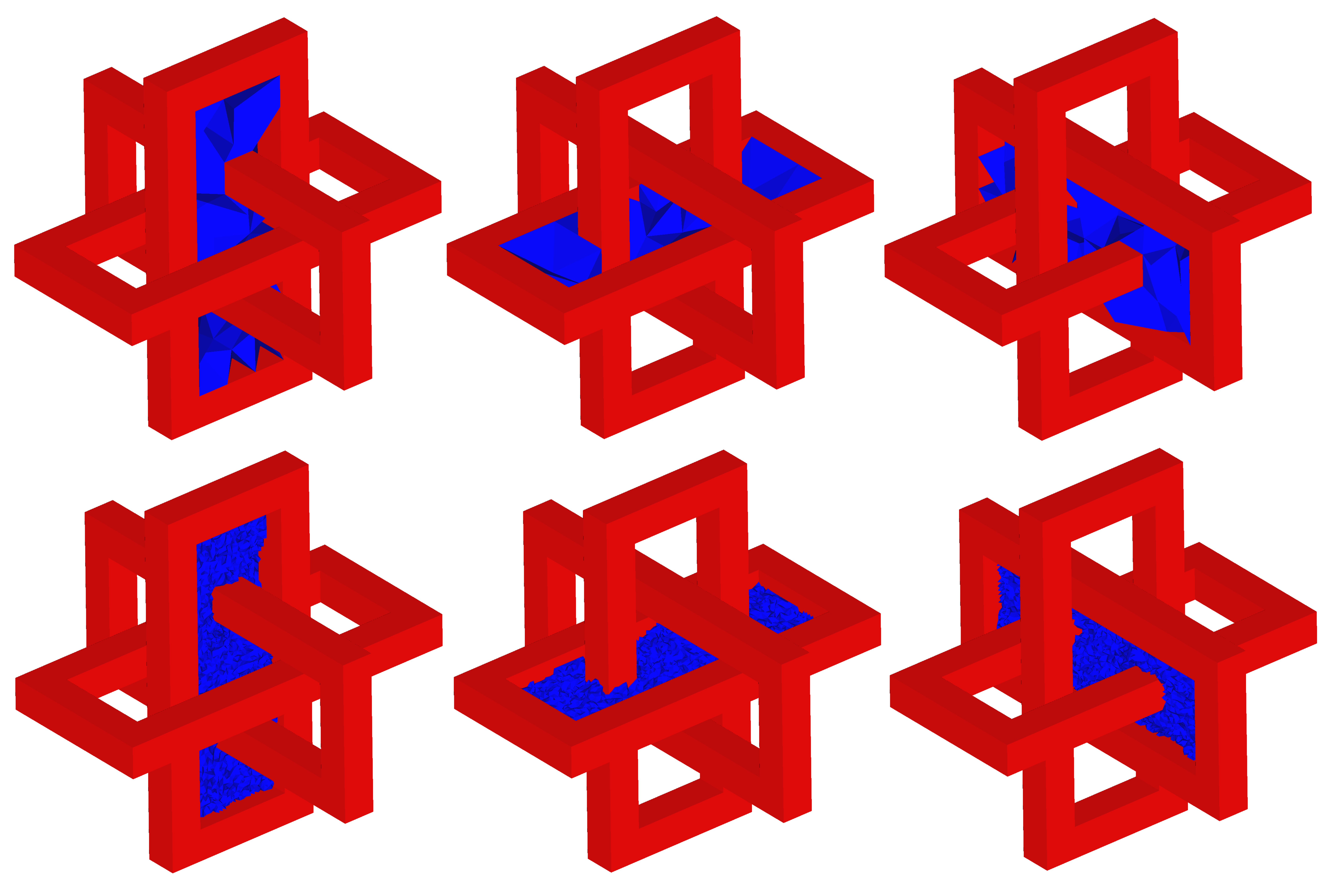}
\caption{The Borromean rings: on the top, three representatives of a basis of $H_2(\overline \Omega, \partial \Omega;\Z)$ for the coarsest mesh. On the bottom, three representatives of a basis of $H_2(\overline \Omega, \partial \Omega;\Z)$ for the finest mesh. The box is not shown for clarity.}
\label{fig:borromean2}
\end{figure}

Table~\ref{tab:borromean} shows the dimension of the four different meshes considered, the computational time and the speed up with respect to GMSH. As can be seen, for the coarsest mesh the speedup is 3 and it increase when considering bigger meshes.
\begin{table}[!h]
\centering
\rowcolors{2}{gray!25}{white}
\begin{tabular}{lcccc}
\rowcolor{gray!50}
Benchmark Borromean rings & Mesh 1 & Mesh 2 & Mesh 3 & Mesh 4\\
%Vertices                & 4692  & 33147  & 246102  & 1647591\\
Edges                   & 29003 & 214807 & 1640732 & 11139998\\
Faces                   & 46723 & 355752 & 2760283 & 18870406\\
%Tetrahedra              & 22411 & 174091 & 1365652 & 9377998\\
%Boundary Vertices       & 1903  & 7572   & 28981   & 114412 \\
%Boundary Edges          & 5703  & 22710  & 86937   & 34323  \\
%Boundary Faces          & 3802  & 15140  & 57958   & 228820 \\
Mesh pre-processing [s]            & 0.300 & 2.530 & 21.53 & 167.1 \\
Hiptmair--Ostrowski [s]            & 0.020 & 0.080 & 0.410 & 2.165 \\
Boundary retrieval [s]             & 0.010 & 0.010 & 0.030 & 0.183 \\
Elimination algorithm [s]          & 0.030 & 0.320 & 3.460 & 30.98\\\hline
Total Time (this paper) [s]        & 0.360 & 2.940 & 25.43 & 200.4\\
Total Time (GMSH \cite{gmsh}) [s]  & 1.076 & 11.19 & 121.1 & $>2$ hours\\
Speedup                             & 3.0  & 3.8   & 4.8    & $-$\\
\end{tabular}
 \caption{The Borromean rings: the number of geometric elements of the triangulation and the computational time.}\label{tab:borromean}
\vspace{-.2cm}
\end{table}

In the next two examples  the domain is the complement with respect to a two-torous of a trefoil knot (Example 3) and the Hopf link (Example 4).

In Example 3 the boundary of the domain has 2 connected components and its first Betti number is $3$ . The number of linking number to be computed is $9+3(4+1)=24$.
In Figure~\ref{fig:trefoil2} we show the three generators of $H_2(\overline \Omega, \partial \Omega; \Z)$ for the trefoil benchmark  and again,
 in Table~\ref{tab:trefoil} we give the dimension of the four different meshes considered, the computational time and the speed up with respect to GMSH with results similar to the previous examples.
\begin{figure}[!h]
\centering
\includegraphics[width=.8\columnwidth]{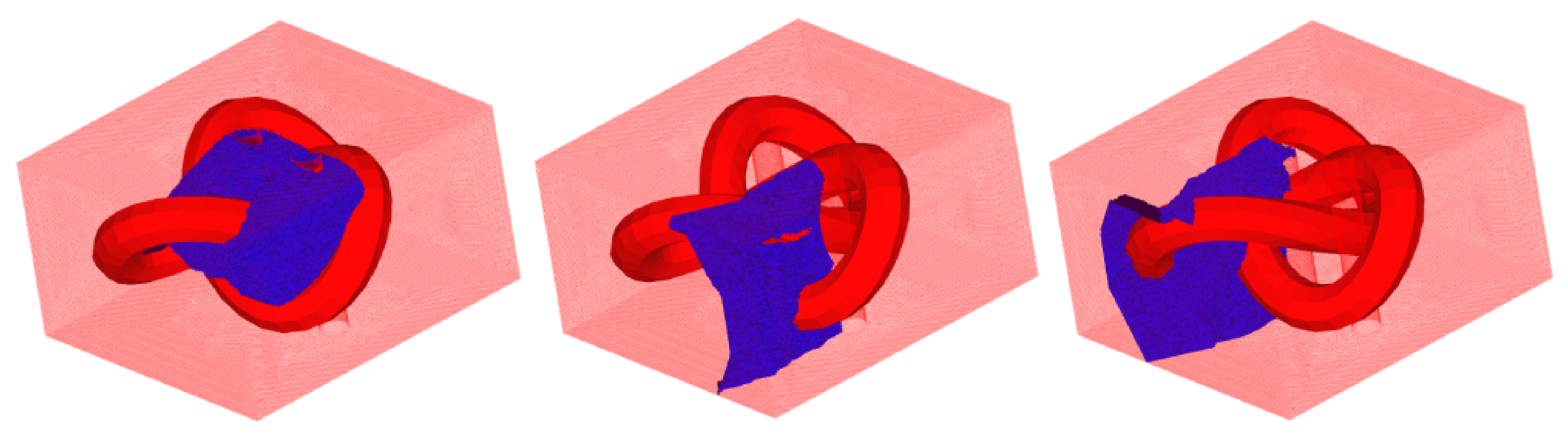}
\caption{The trefoil knot benchmark. Representatives of a basis of $H_2(\overline\Omega, \partial \Omega; \Z)$ generators for the finest mesh are shown.}
\label{fig:trefoil2}
\end{figure}
\begin{table}[!h]
\centering
\rowcolors{2}{gray!25}{white}
\begin{tabular}{ccccc}
\rowcolor{gray!50}
Benchmark trefoil knot  & Mesh 1 & Mesh 2 & Mesh 3 & Mesh 4\\
%Vertices                & 7337  & 28606  & 195146  & 1541567 \\
Edges                   & 45018 & 176123 & 1260407 & 10264628 \\
Faces                   & 72305 & 283758 & 2086618 & 17305967 \\
%Tetrahedra              & 34625 & 136242 & 1021358 & 8582907 \\
%Boundary Vertices       & 3053  & 11272  & 43900   & 140151 \\
%Boundary Edges          & 9165  & 33822  & 131706  & 420459 \\
%Boundary Faces          & 6110  & 22548  & 87804   & 280306 \\
Pre-processing time [s]            & 0.554 & 2.103 & 16.72 & 153.6\\
Hiptmair--Ostrowski [s]            & 0.046 & 0.163 & 0.767 & 3.099\\
Boundary retrieval [s]             & 0.017 & 0.056 & 0.113 & 0.736\\
Elimination algorithm [s]          & 0.052 & 0.256 & 2.595 & 27.98\\\hline
Total Time (this paper) [s]        & 0.669 & 2.578 & 20.20 & 185.4\\
Total Time (GMSH \cite{gmsh}) [s]  & 1.638 & 8.814 & 94287 & $>2$ hours \\
Speedup                            & 2.5  & 3.4  & 4.7   & $-$\\
\end{tabular}
 \caption{Benchmark trefoil knot: the number of geometric elements of the triangulation and the computational time.}\label{tab:trefoil}
 \vspace{-.2cm}
\end{table}
\begin{figure}[!h]
\centering
\includegraphics[width=.55\columnwidth]{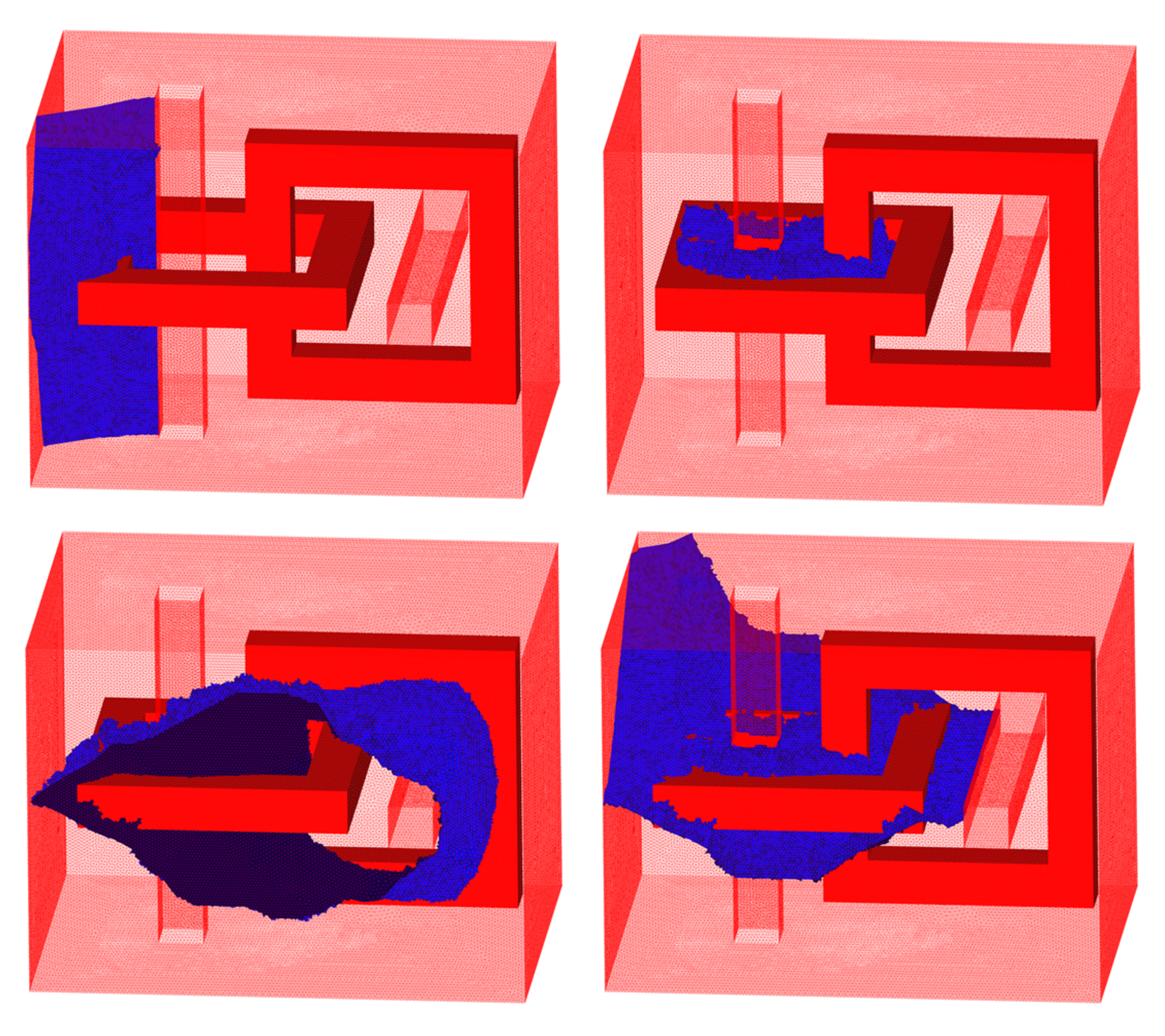}
\caption{The Hopf link benchmark. Representatives of a basis of $H_2(\overline \Omega, \partial \Omega; \Z)$ generators for the finest mesh are shown.}
\label{fig:hopf2}
\end{figure}
In Example 4 the  the domain  is the complement of a Hopf link with respect to a two torous, as illustrated in Figure~\ref{fig:hopf2} where we show four surfaces that are representatives of a basis of $H_2(\overline \Omega, \partial \Omega; \Z)$. In this case the number of connected components of the boundary of the domain is 3, the first Betti number of the domain is 4, and the total number of linking numbers computed is $4^2+3*(4+1+1)=34$. In Table~\ref{tab:trefoil} we report the information about the meshes considered and the computational time. The speed up with respect to GMSH is similar to previous examples.
\begin{table}[!h]
\centering
\rowcolors{2}{gray!25}{white}
\begin{tabular}{ccccc}
\rowcolor{gray!50}
Benchmark Hopf link     & Mesh 1 & Mesh 2 & Mesh 3 & Mesh 4\\
%Vertices                & 6414  & 40804  & 338598  & 1508069\\
Edges                   & 39692 & 263041 & 2255753 & 10152372\\
Faces                   & 64007 & 434513 & 3794183 & 17148224\\
%Tetrahedra              & 30730 & 212277 & 1877029 & 8503922\\
%Boundary Vertices       & 2545  & 9957   & 40123   & 140378\\
%Boundary Edges          & 7641  & 29877  & 120375  & 421140\\
%Boundary Faces          & 5094  & 19918  & 80250   & 280760\\
Mesh pre-processing [s]            & 0.857 & 3.183 & 30.98 & 153.1\\
Hiptmair--Ostrowski [s]            & 0.029 & 0.131 & 0.657 & 3.031\\
Boundary retrieval [s]             & 0.008 & 0.034 & 0.134 & 0.498\\
Elimination algorithm [s]          & 0.044 & 0.415 & 5.118 & 27.82\\\hline
Total Time (this paper) [s]        & 0.938 & 3.763 & 36.89 & 184.5\\
Total Time (GMSH \cite{gmsh}) [s]  & 1.576 & 16.04 & 201.7 & $>2$ hours\\
Speedup                             & 1.7  & 4.3   & 5.5   & $-$\\
\end{tabular}
 \caption{Benchmark Hopf link: the number of geometric elements of the triangulation and the computational time.}\label{tab:hopf}
 \vspace{-.2cm}
\end{table}

As expected, for these four benchmark problems the algorithm proposed in this paper has a linear complexity behaviour as can be seen in Figure~\ref{fig:benchmarks} that illustrate also the speedup with respect to GMSH.
{\small
\begin{figure}[!h]
\begin{minipage}[b]{0.49\linewidth}
\begin{tikzpicture}[scale=0.80]
    \begin{axis}[
            legend pos= south east,
            %x tick label style={/pgf/number format/1000 sep=},
        ]
\addplot[mark=*,mark size=1.5pt, red,mark color=red,] coordinates
{
(76330,1.544)
(227314,5.538)
(2177158,86.284)};
\addplot[mark=square*,mark size=1.5pt, blue,mark color=blue,] coordinates
{
(76330,0.764)
(227314,2.243)
(2177158,21.463)
(17210016,170.095)};
\legend{GMSH,This paper};
\end{axis}
\end{tikzpicture}
\subcaption{Torus with a toric hole benchmark.}
\end{minipage}
\begin{minipage}[b]{0.49\linewidth}
\begin{tikzpicture}[scale=0.80]
    \begin{axis}[
            legend pos= south east,
            %x tick label style={/pgf/number format/1000 sep=},
        ]
\addplot[mark=*,mark size=1.5pt, red,mark color=red,] coordinates
{
(46723,1076/1000)
(355752,11185/1000)
(2760283,121119/1000)};
\addplot[mark=square*,mark size=1.5pt, blue,mark color=blue,] coordinates
{
(46723,360/1000)
(355752,2940/1000)
(2760283,25430/1000)
(18870406,200394/1000)};
\legend{GMSH,This paper};
\end{axis}
\end{tikzpicture}
\subcaption{Borromean rings benchmark.}
\end{minipage}\\
\begin{minipage}[b]{0.49\linewidth}
\begin{tikzpicture}[scale=0.80]
    \begin{axis}[
            legend pos= south east,
            %x tick label style={/pgf/number format/1000 sep=},
        ]
\addplot[mark=*,mark size=1.5pt, red,mark color=red,] coordinates
{
(46723,1638/1000)
(355752,8814/1000)
(2760283,94287/1000)};
\addplot[mark=square*,mark size=1.5pt, blue,mark color=blue,] coordinates
{
(72305,669/1000)
(283758,2578/1000)
(2086618,20199/1000)
(17305967,185428/1000)};
\legend{GMSH,This paper};
\end{axis}
\end{tikzpicture}
\subcaption{Trefoil knot benchmark.}
\end{minipage}
\begin{minipage}[b]{0.49\linewidth}
\begin{tikzpicture}[scale=0.80]
    \begin{axis}[
            legend pos= south east,
            %x tick label style={/pgf/number format/1000 sep=},
        ]
\addplot[mark=*,mark size=1.5pt, red,mark color=red,] coordinates
{
(64007,1576/1000)
(434513,16037/1000)
(3794183,201678/1000)};
\addplot[mark=square*,mark size=1.5pt, blue,mark color=blue,] coordinates
{
(64007,938/1000)
(434513,3763/1000)
(3794183,36885/1000)
(17148224,184477/1000)};
\legend{GMSH,This paper};
\end{axis}
\end{tikzpicture}
\subcaption{Hopf link benchmark.}
\end{minipage}
\caption{Time [s] vs mesh density [number of faces] for the GMSH code and the implementation of the algorithm proposed in this paper.}\label{fig:benchmarks}
\end{figure}
}

%(Figure \ref{tab:borromean}, Fig. \ref{fig:borromean2}, Fig. \ref{fig:benchmarks}b and Fig. \ref{tab:borromean}), the complement of a Hopf link w.r.t. a box (Fig. \ref{fig:hopf2}, Fig. \ref{fig:benchmarks}c and  \ref{tab:hopf}) and the complement of a trefoil knot w.r.t. a solid 2-fold torus (Fig. \ref{fig:trefoil2}, Fig. \ref{fig:benchmarks}d and  \ref{tab:trefoil}).
%

We finally consider an example where the dimension of $H_2(\overline \Omega, \partial \Omega; \Z)$  is much bigger (equal to $128$) consisting in a solid 100-fold torus with eight cavities, see Figure~\ref{fig:100holes}. The cavities are two solid 11-fold tori and six solid tori. So, the number of connected components of the boundary is  $9$ and the first Betti number of the domain is $100+ 22+ 6= 128$.
\begin{figure}[!h]
\centering
\vspace{-2.5cm}
\includegraphics[width=.7\columnwidth]{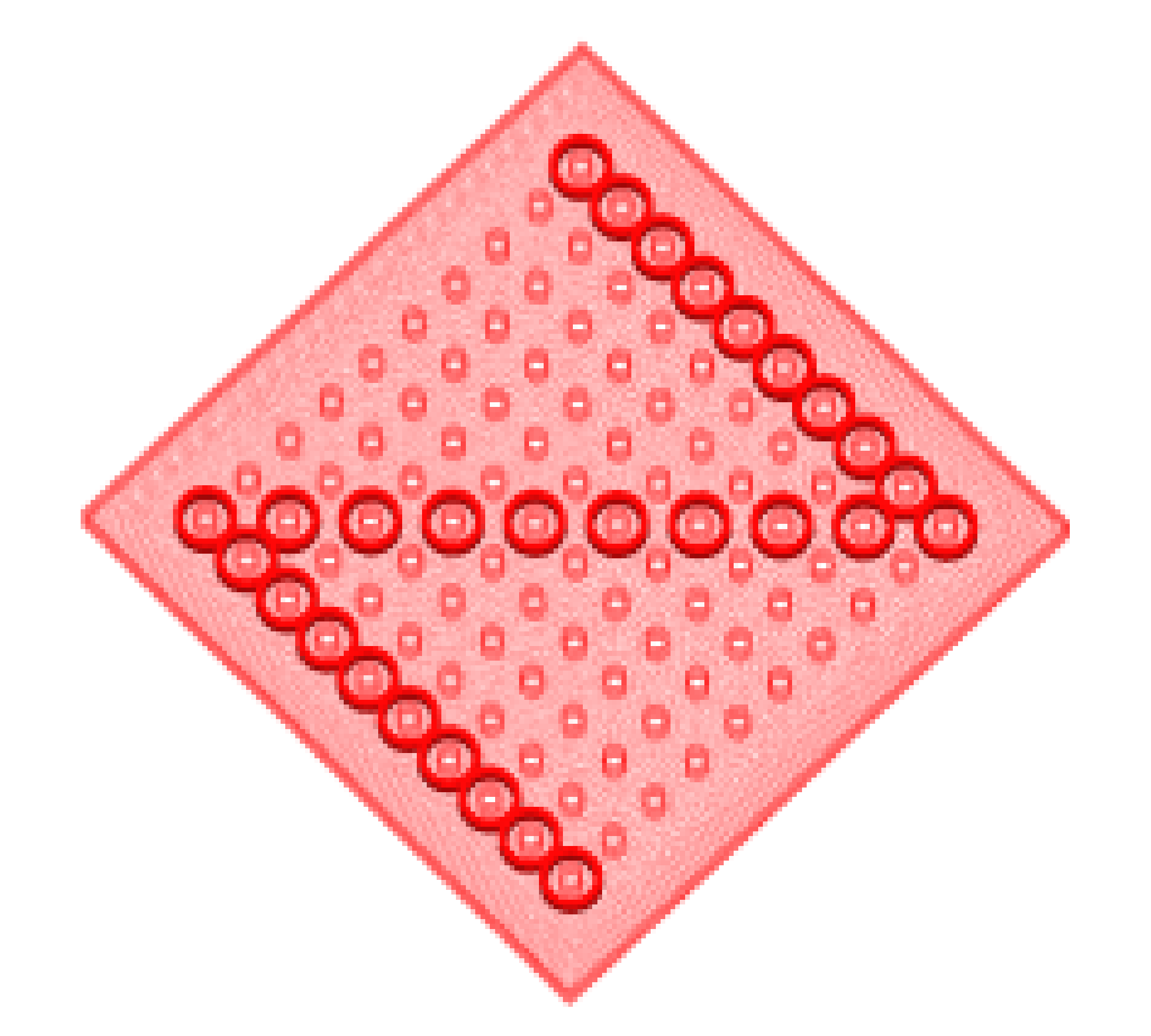}
\vspace{-3cm}
\caption{The plate with holes benchmark.}
\label{fig:100holes}
\end{figure}

In this case the number of linking numbers to compute is huge, equal to $128^2+3 \, (100^2+ 2 \, 11^2+6)=47128$. For this reason in the smaller examples GMSH results faster than the approach proposed in this paper as can be seen in Table~\ref{tab:100holes}. Yet, when the complex cardinality gets into the range of real-life problems, we again get a sensible speedup. In particular, in the last mesh of more than 7 millions edges, GMSH wasn't able to produce a results after 4 hours of wall time, whereas our implementation took less than five minutes.
\begin{table}[!h]
\centering
\rowcolors{2}{gray!25}{white}
\begin{tabular}{ccccc}
\rowcolor{gray!50}
Benchmark plate with holes & Mesh 1 & Mesh 2 & Mesh 3 & Mesh 4\\
%Vertices                & 8620   & 56542  & 183447  & 1172590 \\
Edges                   & 45596  & 334526 & 1164992 & 7740566 \\
Faces                   & 65396  & 523825 & 1908897 & 12956479 \\
%Tetrahedra              & 28539  & 245960 & 927471  & 6388622\\
%Boundary Vertices       & 8080   & 31667  & 53717   & 178997\\
%Boundary Edges          & 24954  & 95715  & 161865  & 537705\\
%Boundary Faces          & 16636  & 63810  & 107910  & 358470\\\hline
Pre-processing time [s]           & 0.493  & 4.102   & 15.79  & 118.6\\
Hiptmair--Ostrowski [s]           & 3.251  & 23.14   & 17.60  & 50.66\\
Boundary retrival [s]             & 1.458  & 19.14   & 19.97  & 39.11\\
Elimination algorithm [s]         & 0.198  & 1.789   & 6.931  & 55.95\\ \hline
Total Time (this paper) [s]       & 5.400  & 48.17   & 60.30  & 264.3\\
Total Time (GMSH \cite{gmsh}) [s] & 2.044  & 27.86   & 138.1  & $> 4$ hours\\
Speedup                           & 0.38   & 0.58    & 2.3    & $-$\\
\end{tabular}
 \caption{Benchmark plate with holes: the number of geometric elements of the triangulation and the computational time.}\label{tab:100holes}
\end{table}

%
%\begin{figure}[!t]
%\centering
%\includegraphics[width=5.cm]{borromean.pdf}
%\caption{The Borromean rings benchmark. The domain is the complement of the Borromean rings with respect to a box. The box is only outlined in the picture for the sake of clarity.}
%\label{fig:borromean}
%\end{figure}

In Figure~\ref{fig:100holesg} we can see that in this benchmark problem the time on small examples is dominated by the linking number computations so it is not strictly linear.

\begin{figure}
\begin{minipage}[b]{0.49\linewidth}
\begin{tikzpicture}[scale=0.90]
    \begin{axis}[
            legend pos= south east,
            %x tick label style={/pgf/number format/1000 sep=},
        ]
\addplot[mark=*,mark size=1.5pt, red,mark color=red,] coordinates
{
(64007,2.044)
(434513,27.86)
(3794183,138.1)};
\addplot[mark=square*,mark size=1.5pt, blue,mark color=blue,] coordinates
{
(65396,5.400)
(523825,48.17)
(1908897,60.30)
(12956479,264.3)};
\legend{GMSH,This paper};
\end{axis}
\end{tikzpicture}
\subcaption{Plate with holes benchmark.}
\end{minipage}
\begin{minipage}[b]{0.49\linewidth}
\begin{tikzpicture}[scale=0.90]
    \begin{axis}[
            legend pos= south east,
            %x tick label style={/pgf/number format/1000 sep=},
        ]
\addplot[mark=*,mark size=1.5pt, red,mark color=red,] coordinates
{
(16636,2.044)
(63810,27.86)
(107910,138.1)};
\addplot[mark=square*,mark size=1.5pt, blue,mark color=blue,] coordinates
{
(16636,5.400)
(63810,48.17)
(107910,60.30)
(358470,264.3)};
\legend{GMSH,This paper};
\end{axis}
\end{tikzpicture}
\subcaption{Plate with holes benchmark.}
\end{minipage}
\caption{Time [s] vs mesh density [on the left number of faces, on the right number of faces on the boundary] for the GMSH code and the implementation of the algorithm proposed in this paper.}\label{fig:100holesg}
\end{figure}
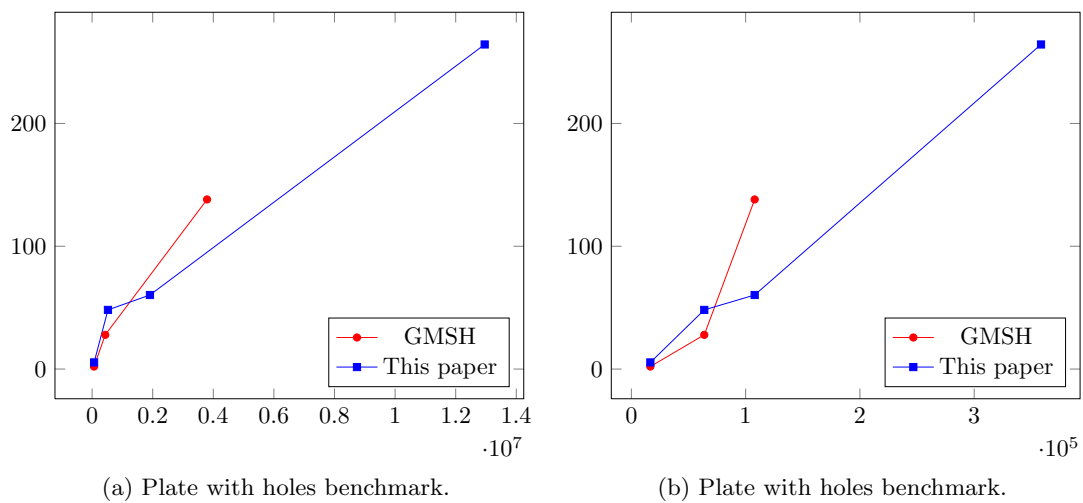

%%%%%%%

%
%
%\section{Computation of a basis of $H_2(\overline \Omega, \partial \Omega; \Z)$} \label{sec:numerical}
%
%\begin{itemize}
%\item
%A torus with a concentric toroidal cavity (dim($H_2(\overline \Omega, \partial \Omega; \Z)$)=2).
%
%Two trivial linear systems $1 \times 1$.
%
%\item
%A torus with a knotted concentric cavity (dim($H_2(\overline \Omega, \partial \Omega; \Z)$)=2).
%
%Two trivial linear systems $1 \times 1$.
%
%\item
%A cube without the Borromean rings (dim($H_2(\overline \Omega, \partial \Omega; \Z)$)=3).
%
%Two trivial linear systems $1 \times 1$.
%
%
%
%\item
%A cube without the Hopf link (dim($H_2(\overline \Omega, \partial \Omega; \Z)$)=2).
%
%In fact $\sigma'_m =\sigma_m$ $m=1,2,3$ because in the three $1 \times 1$ linear systems that we solve to compute $\sigma'_m$ from $\sigma_m$ the right hand term is zero.
%
%
%\item
%A two torus, without a three torus (dim($H_2(\overline \Omega, \partial \Omega; \Z)$)=5).
%
%Two $3 \times 3$ linear systems and three $2 \times 2$ linear systems?
%
%
%\item
%A torus with a concentric toroidal cavity and another toroidal cavity around the previous one (like the first example with a new toroidal cavity) (dim($H_2(\overline \Omega, \partial \Omega; \Z)$)=3).
%
%Six trivial linear systems $1 \times 1$.
%
%
%\end{itemize}
%
%How can we check that the computed surfaces are correct?????
%
%
%%%%%%%

\bibliographystyle{siam}
\bibliography{BaseH2-2}

\end{document}